\documentclass[10pt]{amsart}
\usepackage{amssymb}
\usepackage{bm}
\usepackage{graphicx}
\usepackage[centertags]{amsmath}
\usepackage{amsfonts}
\usepackage{amsthm}
\usepackage{graphicx}

\newtheorem{thm}{Theorem}
\newtheorem{cor}[thm]{Corollary}
\newtheorem{lem}[thm]{Lemma}

\newtheorem{prop}[thm]{Proposition}

\newtheorem{defn}[thm]{Definition}
\theoremstyle{definition}
\newtheorem{rem}{Remark}

\newcommand{\rr}{\mathbb{R}}
\newcommand{\nn}{\mathbb{N}}
\newcommand{\ee}{\varepsilon}
\newcommand{\dd}{\delta}

\begin{document}

\title[Hausdorff  measures and functions of bounded quadratic variation]{Hausdorff  measures and functions of
 bounded quadratic variation}
\author{D. Apatsidis, S.A. Argyros and V. Kanellopoulos}
\address{National Technical University of Athens, Faculty of Applied Sciences,
Department of Mathematics, Zografou Campus, 157 80, Athens,
Greece} \email{dapatsidis@hotmail.com, sargyros@math.ntua.gr,
bkanel@math.ntua.gr} \footnotetext[2]{Research
supported by PEBE 2007.} 


\begin{abstract} To each function $f$ of bounded quadratic variation we
associate a Hausdorff measure $\mu_f$. We show that the map
$f\to\mu_f$ is locally Lipschitz and onto the positive cone of
$\mathcal{M}[0,1]$. We use the measures $\{\mu_f:f\in V_2\}$ to
determine the structure of the subspaces of $V_2^0$ which either
contain $c_0$ or the square stopping time space $S^2$.
\end{abstract}

\maketitle
\section{Introduction.}
 The functions of bounded quadratic variation,
introduced by N. Wiener in \cite{W},  have been extensively
studied  in their own right as well as for their applications. For
example, related results can be found in \cite{BLS1}, \cite{BLS2},
\cite{G}, \cite{P-W} and also in the monograph \cite{DN} where
several applications are included.

Our intention in the present work is to study the structure of the
subspaces of $V_2^0$. In the sequel we shall denote by $V_2$ the
space of all real valued functions $f$ with  bounded quadratic
variation, defined on the unit interval and satisfying $f(0)=0$.
The space $V_2$ endowed with the quadratic variation norm is a
Banach space. The aforementioned space  $V_2^0$ is a separable
subspace of $V_2$ of significant importance; it is defined as  the
closed subspace of $V_2$ containing all the square absolutely
continuous functions, a concept introduced by R. E. Love in the
early 50's (cf. \cite{L}).

The space $V_2^0$ was introduced  by S. V. Kisliakov in \cite{K}
as an isometric version of the Lindenstrauss' space $JF$. His aim
was to provide easier proofs of the fundamental  properties of
$JF$. The space $V_2^0$ is  separable, not containing $\ell_1$ and
with non-separable dual. These properties were the most
distinctive ones for $JF$, as such a space answers  in the
negative a problem posed by S. Banach. Earlier R.C. James \cite{J}
had presented the James Tree ($JT$) space which is the analogue of
$V_2^0$ in the frame of the sequence spaces. It is notable that
the class of the separable Banach spaces not containing $\ell_1$
and with non-separable dual, which  appears as an exotic subclass
of Banach spaces, includes spaces like $V_2^0$ naturally arising
from other branches of Analysis.

There exists a fruitful relation between the spaces $V_2^0$ and
$V_2$ pointed out by Kisliakov (cf.  \cite{K}). Namely $V_2$
naturally coincides with the second dual of $V_2^0$ and moreover
the w$^*$-topology on the bounded subsets of $V_2$ coincides with
the topology of the pointwise convergence. Among the consequences
of the preceding remarkable property is that every $f\in V_2$ is
the pointwise limit of a bounded sequence $(f_n)_n$ from $V_2^0$
(cf. \cite{OR}). The variety of the classical Banach spaces which
are isomorphic to a subspace of $V_2^0$ is large and rather
unexpected. Indeed, beyond the space $\ell_2$ which among the
classical spaces, is the most relative to $V_2^0$, as it was
stated  in \cite{LS}, the space $c_0$ is also isomorphic to a
subspace of $V_2^0$. Moreover for all $2\leq p<+\infty$, the space
$\ell_p$ shares the same property (cf. \cite{B}).

In our previous work (cf.\cite{AAK}), but also in  \cite{AMP}, we
have started studying the structure of the subspaces  of $V_2^0$.
>From our point of view a sufficient understanding of the structure
of $V_2^0$ requires answers to the following problems.

 \textit{Problem 1.} Let $X$ be a reflexive subspace of
$V_2^0$. Does there exist some $2\leq p<\infty$ such that $\ell_p$
is isomorphically embedded into $X$?

 As we have mentioned earlier  all $\ell_p$,  $2\leq p<\infty$, are
embedded into $V_2^0$. Also,  in \cite{B} it was shown that no
$\ell_p$ $1\leq p< 2$ is isomorphic to a subspace of $V_2^0$. It
is worth pointing out that the embedding of $\ell_p$ $2\leq
p<\infty$ is rather indirect and uses the space $S^2$ which is one
of the central objects of the present paper. The space $S^2$ is
the quadratic stopping time space and is the square
convexification of the space $S^1$. The latter space was defined
by H. P. Rosenthal as the unconditional analogue of the space
$L^1(\lambda)$. Both spaces (i.e. $S^1$, $S^2$) belong to the
wider class of the spaces $S^p$, $1\leq p<\infty$ which we are
about to define. We denote by $2^{<\nn}$ the dyadic tree and by
$c_{00}(2^{<\nn})$ the vector space of all real valued functions
defined on $2^{<\nn}$ with finite support. For $1\leq p<\infty$ we
define the $\|\cdot\|_{S^p}$ on $c_{00}(2^{<\nn})$ as follows. For
$x\in c_{00}(2^{<\nn})$, we set
\[\|x\|_{S^p}=\sup\Big(\sum_{s\in A}|x(s)|^p\Big)^{1/p}\] where
the supremum is taken over all antichains $A$ of $2^{<\nn}$. The
space $S^p$ is the completion of $(c_{00}(2^{<\nn}),
\|\cdot\|_{S^p})$. As we mentioned above the space $S^1$ (the
stopping time space) was defined by Rosenthal and the spaces $S^p$
($1<p<\infty$) appeared in S. Buechler's Ph.D. Thesis \cite{B}.
The space $S^1$  has an unconditional basis and G. Schechtman, in
an unpublished work, showed that it contains all $\ell_p$ $1\leq
p<\infty$. This result was extended in \cite{B} to all $S^p$
spaces by showing that for every $p\leq q$, $\ell_q$ is embedded
into $S^p$. An excellent and detailed study of the stopping time
space $S^1$, in fact in a more general setting, is included in N.
Dew's Ph.D. Thesis \cite{D}. The interested reader will also find
there,  among other things, a proof of Schechtman's unpublished
result. Let us also point out that the analogous problem to
Problem 1 for the spaces $S^p,\ 1\leq p<\infty$ remains also open.
An important result in \cite{B} shows that $S^2$ is isomorphic to
a subspace of $V^0_2$ and this actually yields that $V_2^0$
contains isomorphs of all $\ell_p,\ 2\leq p<\infty$. Before
closing our discussion for Problem 1, let us also note that for
every infinite chain $C$ of $2^{<\nn}$ the subspace of $S^2$
generated by $\{e_s:s\in C\}$ is isomorphic to $c_0$, while for
every infinite antichain $A$ the corresponding one is isomorphic
to $\ell_2$. Thus, if a subspace $X$ of $V_2^0$ contains an
isomorph of $S^2$ then it contains all possible classical spaces
that are embedded to $V_2^0$.

Our next two  problems concern non-reflexive subspaces of $V_2^0$.
Let us begin with a result from \cite{AMP} which asserts that
every non reflexive subspace $X$ of $V_2^0$ contains an isomorph
of $\ell_2$ or $c_0$. To see this we start with some $f\in
X^{**}\setminus X$, where $X^{**}$ is considered as a subspace of
$V_2$. Since $X$ does not contain $\ell_1$ Odell-Rosenthal's
theorem (cf. \cite{OR}) yields that there exists a bounded
sequence $(f_n)_n$ in $X$ pointwise converging to $f$. If $f$ is
discontinuous then there exists a sequence $(g_k)_k,\
g_k=f_{n_k}-f_{m_k}$ equivalent to the $\ell_2$ basis and hence
$\ell_2$ is embedded into $X$. The case of a continuous $f$ is
more interesting. As  it is shown in \cite{AMP}, such an $f$ is a
difference of bounded semicontinuous functions (DBSC) when $f$ is
considered as a function with domain $(B_{(V_2^0)^*},w^*)$. A
result of Haydon, Odell and Rosenthal (cf. \cite{HOR}), yields
that the sequence $(f_n)_n$ has a convex block subsequence
$(g_n)_n$ equivalent to the summing basis of $c_0$. Let us note
that the existence of a continuous function $f$ in
$X^{**}\setminus X$ is actually equivalent to the embedding of
$c_0$ into $X$.

 The second problem concerns subspaces of $V_2^0$
with non- separable dual and it is stated as follows.

\textit{Problem 2.} Is it true that every subspace $X$ of $V_2^0$
with $X^{*}$ non separable  contains an isomorph of $V_2^0$
itself? Moreover, is every complemented subspace $X$ of $V_2^0$
with non separable dual isomorphic to $V_2^0$?

An affirmative answer to the second part of Problem 2, yields that
$V_2^0$ is a primary space. In \cite{AAK} it has been shown that
the corresponding problem to the preceding one in James' space
$JT$ has an affirmative answer, an evident supporting the
possibility for a positive solution to Problem 2. It is worth
mentioning that as is shown in \cite{AAK}, every subspace $X$ of
$V_2^0$ with non-separable dual contains the space $TF$. The space
$TF$ is a sequence space with non separable dual, introduced in
\cite{AAK}. It is isomorphic to any subspace of $V_2^0$ generated
by a tree family $(f_s)_{s\in 2^{<\nn}}$ of trapezoids. The latter
spaces were considered in \cite{B}, for showing that $V_2^0$ does
not contain isomorphs of $JT$. In \cite{AAK} it is also  stated
without proof that the space $S^2$ is embedded into $TF$ which, as
we have previously mentioned, yields that every subspace of
$V_2^0$ with non-separable dual contains isomorphs of all possible
classical spaces that are embedded in $V_2^0$. In the present
paper we give a proof of the embedding of $S^2$ into $TF$,
granting that $c_0$ is embedded into $TF$ from \cite{B}.

The third problem concerns subspaces of $V_2^0$ with non separable
second dual and it is stated as follows.

\textit{Problem 3.} Is it true that every subspace $X$ of $V_2^0$
with $X^{**}$ non separable contains $c_0$?

The main goal of the present work is to provide a positive
solution to this problem. Before start explaining our solution, we
point out that the preceding results on subspaces $X$ with non
separable dual reduce the problem to those $X$ with $X^*$
separable and $X^{**}$ non separable. Also, as we noted above, the
embedding of $c_0$ into $X$ is equivalent to the existence of a
function $f\in (X^{**}\setminus X)\cap C[0,1]$. In the early
stages of our engagement to this problem, we observed that when
$X^*$ is separable, the set $D_{X^{**}}=\{t\in[0,1]: \exists f\in
X^{**}, oscf(t)>0\}$ is at most countable, a fact supporting an
affirmative solution to the problem. However we had no further
progress, until the moment where we discovered a new concept which
plays a key role to our approach. This is a Hausdorff type measure
$\mu_f$ associated to every $f\in V_2$. The measure $\mu_f$ is
defined as follows. First we introduce some notation.

 Given
$f:[0,1]\to\rr$ and $\mathcal{P}=\{ t_0<\ldots<t_p\}\subseteq
[0,1]$, with $p\geq 1$, let
$\|\mathcal{P}\|_{\max}=\max\{t_{i+1}-t_i:0\leq i\leq p-1\}$ and
$\textit{v}_2^2(f,\mathcal{P})=\sum_{i=0}^{p-1}(f(t_{i+1})
-f(t_{i}))^2$.
 For every  $f\in V_2$ and  for every
interval $I$ of $[0,1]$ we set
\[\widetilde{\mu}_f(I)=\inf_{\delta>0}\widetilde{\mu}_{f,\delta}(I),\]
where for each $\delta>0$,
$\widetilde{\mu}_{f,\delta}(I)=\sup\{\textit{v}_2^2(f,\mathcal{P}):
\mathcal{P}\subseteq I
\;\text{and}\;\|\mathcal{P}\|_{\max}<\dd\}$.\\ The collection
$\{\widetilde{\mu}_f(I): I\;\text{is an interval of}\;[0,1]\}$
defines an outer measure and $\mu_f$ is the regular measure
induced by $\widetilde{\mu}_{f}$ on the Borel subsets of $[0,1]$.
We should mention that N. Wiener himself had also considered the
quantity $\sqrt{\widetilde{\mu}_f[0,1]}$, pointing out that it is
a seminorm on $V_2$. The measure $\mu_f$ incorporates a sufficient
amount of information concerning the function $f$. Thus $\mu_f=0$
if and only if $f\in V_2^0$, $\mu_f$ is continuous (diffuse) if
and only if $f$ is continuous and also the discrete (atomic) part
of $\mu_f$ is supported by the points of discontinuity of $f$.
Furthermore the following hold.

\textit{Proposition 1.} Let $f\in V_2$. Then the set of the points
of differentiability of $f$ has $\mu_f$-measure zero.

As a consequence we obtain the following.

\textit{Corollary.} Let $f$ be a continuous function in $V_2$. If
the set of all non differentiability points of $f$ is at most
countable then $f$ belongs to $V_2^0$. Moreover if $f\in
(V_2\setminus V_2^0)\cap C[0,1]$ then the set of all non
differentiability points of $f$ contains a perfect set.

The second result concerns the variety of the elements of $V_2$.

\textit{Proposition 2.} Let $\Phi: V_2\to \mathcal{M}^+[0,1]$ be
the function that maps $f$ to $\mu_f$. Then $\Phi$ is locally
Lipschitz and onto. In particular for every continuous positive
measure $\mu\in\mathcal{M}^+[0,1]$ there exists $f\in
(V_2\setminus V_2^0)\cap C[0,1]$ such that $\mu_f=\mu$.

The measure $\mu_f$ has a central role in the solution of Problem
3. In particular the following inequality is the key ingredient.
For every $f\in V_2$ the following holds.
 \begin{equation}\label{basicineq}\sqrt{\|\mu_f\|}\leq dist (f, V_2^0)\leq
\|\widetilde{osc}_\mathcal{K}
f\|_{\infty}\leq\sqrt{\|\mu_f\|}+2\sqrt{\|\mu_f^d\|}\end{equation}
where $\mathcal{K}$ is a w$^*$-closed subset of $B_{(V_2^0)^*}$,
$1$-norming $V_2^0$, $\widetilde{osc} f$ is as in \cite{R} or
\cite{AK} and was introduced in \cite{KL} and also $\mu_f^d$ is
the discrete part of $\mu_f$. Note that when $f$ is continuous
then the inequality (\ref{basicineq}) becomes equality and hence
$dist (f, V_2^0)=\sqrt{\|\mu_f\|}$. Furthermore the measures
$\{\mu_f:f\in V_2\}$ permit us to have a better and more precise
understanding of the structure of $X$ when $X^{**}$ is non
separable. Thus we prove the following.

\textit{Theorem.} Let $X$ be a closed subspace of $V_2^0$. Then
the following hold.
\begin{enumerate}\item The space $X$
contains an isomorphic copy of  $c_0$  if and only if $X^{**}$ is
non separable.\item The space $X$ contains an isomorphic copy of
$S^2$  if and only if $\mathcal{M}_{X^{**}}=\{\mu_f:f\in X^{**}\}$
is non-separable.
\end{enumerate}

Note that when $X^*$ is non separable the stronger case (case (2))
of the above theorem occurs. When   $X$ is isomorphic to $c_0$
then $X^*$ is separable and $\mathcal{M}_{X^{**}}$ is separable.
On the other hand, any subspace $X$ of $V_2^0$ isomorphic to $S^2$
is an example of a subspace $X$ with separable dual and
$\mathcal{M}_{X^{**}}$ non separable.

In the rest of the introduction we shall describe the basic steps
towards a proof of the main theorem. Let us start by saying that a
function $h\in V_2^0$ is $(C,\ee)-$dominated by a measure
$\mu\in\mathcal{M}^+[0,1],$ if for every finite family
$\mathcal{I}=([a_i,b_i])_{i=1}^n$ of non overlapping intervals it
holds
\[\sum_{i=1}^n\Big(h(b_i)-h(a_i)\Big)^2\leq C\mu(\cup\mathcal{I})+\ee\]
This domination property permits us to engage measures with
sequences $(h_n)_n$ which are equivalent to the usual basis of
$c_0$ as follows.

\textit{Proposition 3.} Let $(h_n)_{n}$
 be  a seminormalized sequence of functions of $V_2^0$,
$(\ee_n)_n$ be a null sequence of positive real numbers and
$\mu\in \mathcal{M}^+[0,1]$ such that for some $C>0$ each $h_n$ is
$(C,\ee_n)-$ $\mu$ dominated   and  $\lim_n\|h_n\|_{\infty}= 0$.
Then there is a subsequence of $(h_n)_{n}$ equivalent to the usual
basis of $c_0$.

The next result explains how we pass data from an $f\in
X^{**}\setminus X$ to elements of $X$ itself.

\textit{Proposition 4.} Let $X$ be a subspace of $V_2^0$, $f\in
X^{**}\setminus X$ and $(f_n)_n$ be a bounded sequence in $X$
pointwise convergent to $f$.  Then for every $0<\delta<dist(f,X)$
and
 every sequence $(\ee_n)_n$ of positive real numbers there exists
a  convex block sequence $(h_n)_n$ of  $(f_n)_n$ such that for all
$n<m$ the following properties are satisfied.
\begin{enumerate}
\item [(i)] $\delta< \|h_m-h_n\|_{V_2}\leq 2M$, where $M=\sup_n
\|f_n\|_{V_2} $. \item[(ii)] $\|h_m-h_n\|_{\infty}\leq
2\|\widetilde{osc}_{[0,1]} f\|_{\infty}+\ee_n\leq
4\|f\|_\infty+\ee_n.$  \item [(iii)] $h_m- h_n$ is
$(4,\widetilde{\ee}_n)$-$\mu_f$ dominated, where
$\widetilde{\ee}_n=32\|f\|_{V_2}\sqrt{\|\mu_f^d\|}+\ee_n$.
\end{enumerate}

The proof of the proposition uses inequality (\ref{basicineq}) and
also optimal sequences pointwise convergent to the function $f$
(cf. \cite{AK}). Note that if we additionally assume that $f$ is
continuous, in which case $\mu_f^d=\widetilde{osc}_{[0,1]}f=0$,
Propositions 3 and 4 almost immediately yield that the space $X$
contains $c_0$ , a result initially proved with a different method
in \cite{AMP}.

The proof of the  main result is divided into two cases. In the
first case we consider subspaces $X$ of $V_2^0$ with $X^*$
separable, $X^{**}$ non separable  and $\mathcal{M}_{X^{**}}$
separable. Then using Proposition 4 and the separability of
$X^{*}$, we may select a seminormalized sequence $(H_n)_n$ in
$V_2^0$ and a norm converging sequence of measures $(\mu_n)_n$ in
$\mathcal{M}_{X^{**}}$
 such that
 each $H_n$  is $(4, \ee_n)-\mu_n$
 dominated. Then an easy modification of Proposition 3 yields that there exists a
 subsequence of $(H_n)_{n}$  equivalent to $c_0$ basis.

The second case, namely when $\mathcal{M}_{X^{**}}$ is non
separable, is  more involved. Here, we first give  sufficient
conditions for the embedding of the space  $S^2$ into a subspace
$X$ of $V_2^0$ . Moreover, using again Proposition 4, we construct
a seminormalized tree family  of functions $(H_s)_{s\in 2^{<\nn}}$
of elements of $X$ and a bounded family of measures $(\mu_s)_{s\in
2^{<\nn}}$. For each $s\in 2^n$ we define a finite subset
$L_s\subseteq 2^{2n}$ with $card(L_s)=2^n$ and we set
$G_s=2^{-n/2}\sum_{t\in L_s}H_t$ and $\nu_s=2^{-n}\sum_{t\in
L_s}\mu_{t}$. The proof ends by proving that these new tree
families satisfy the requirements
 for containing a tree subfamily equivalent to the  $S^2$ basis.

We consider the present work as a  step towards the understanding
of the structure of $V_2^0$. Our approach has revealed a new
component, the Hausdorff measure $\mu_f$ associated to a function
$f$  of bounded quadratic variation, which is of independent
interest and could be useful to a further investigation of $V_2$
as well as in applications.

 We close this introduction by pointing
out that all the results contained here remain valid under obvious
modifications for the space $V_p^0$, for all $1<p<\infty$.

\section{Preparatory work on $V_2$.} This section is divided into three  subsections. First we fix the
 notation that we shall use. In the second subsection  we prove  that  the set
  of discontinuity points of the elements of $X^{**}$ when  $X^{*}$ is separable is countable
   and also in this case $(X^{*},\|\cdot\|_\infty)$ is separable.
   Finally,
   we introduce  the biorthogonal
   families of functions of $V_2^0$. Such  families  share
   nice properties and as we will see  they play a critical role in the proofs of almost all of our  results.
\subsection{Preliminaries.}
 We start with the notation which concerns intervals
as well as families of intervals of $[0,1]$. The length of an
interval $I$ will be denoted by $|I|$.   For a finite family
$\mathcal{I}$ of intervals,
$\|\mathcal{I}\|_{\max}=\max\{|I|:I\in\mathcal{I}\}$ and
$\|\mathcal{I}\|_{\min}=\min\{|I|:I\in\mathcal{I}\}$.

 By $\mathcal{A}$  we
denote the set of all finite families of intervals of $[0,1]$ with
pairwise disjoint interiors.  A sequence $(\mathcal{I}_i)_{i}$ in
$\mathcal{A}$, will be called \textit{disjoint}, if for every
$i\neq j$, $I\in\mathcal{I}_i$ and $J\in\mathcal{I}_j$, the
interiors of $I$ and $J$ are disjoint.
 Also by
$\mathcal{F}$ we denote the set of all finite families of pairwise
disjoint closed intervals of $[0,1]$. More generally, given a
subset $S\subseteq [0,1]$, $\mathcal{F}(S)$ is the set of all
$\mathcal{I}\in\mathcal{F}$ such that the endpoints of  every
$I\in \mathcal{I}$ belong to $S$.

 Given $f:[0,1]\to\rr$ and
$\mathcal{P}=\{ t_0<\ldots<t_p\}\subseteq [0,1]$, with $p\geq 1$,
the quadratic variation of $f$ on $\mathcal{P}$ is the quantity
\[\textit{v}_2(f,\mathcal{P})=\Big(\sum_{i=0}^{p-1}(f(t_{i+1})
-f(t_{i}))^2\Big)^{1/2}\] Similarly for a
$\mathcal{I}=(I_i)_{i=1}^k$ in $\mathcal{A}$, we set
$\textit{v}_2(f,\mathcal{I})=(\sum_{i=0}^{k}(f(b_i)
-f(a_i))^2)^{1/2},$ where for each $1\leq i\leq k$, $a_i,b_i$ are
the endpoints of $I_i$ (if $\mathcal{I}$ is the empty sequence
 then we define $\textit{v}_2(f,\varnothing)=0$). The quantity
 $\textit{v}_2(f,\mathcal{I})$ has also been defined in \cite{P-W}
 where  the exponent $1/2$ is omitted.

 Notice that every   $\mathcal{P}$ as
above, determines the  family
$\mathcal{I}_\mathcal{P}=((t_i,t_{i+1}))_{i=0}^{p-1}$ in
$\mathcal{A}$ and
$\textit{v}_2(f,\mathcal{P})=\textit{v}_2(f,\mathcal{I}_\mathcal{P})$.
It is easy to see that for every $f,g:[0,1]\to \rr$ and every
$\mathcal{I}\in\mathcal{A}$, we have that
\begin{equation}\label{P1} |\textit{v}_2(f,\mathcal{I})-\textit{v}_2(g,\mathcal{I})|\leq
\textit{v}_2(f+g, \mathcal{I})\leq
\textit{v}_2(f,\mathcal{I})+\textit{v}_2(g,\mathcal{I})
\end{equation}
Moreover for every  disjoint partition
$\mathcal{I}=\cup_i\mathcal{I}_i$ of a $\mathcal{I}\in
\mathcal{A}$,
\begin{equation}\textit{v}_2(f,\mathcal{I})\leq\sum_i\textit{v}_2(f,\mathcal{I}_i)\;\;\;\text{and}\;\;\;
\textit{v}_2^2(f,\mathcal{I})=\sum_i\textit{v}_2^2(f,\mathcal{I}_i)\end{equation}

For $\widetilde{\mathcal{I}}, \mathcal{I}$ in $\mathcal{A}$, we
write $\widetilde{\mathcal{I}}\preceq\mathcal{I}$ if for every
$\widetilde{I}\in\widetilde{\mathcal{I}}$ there is
$I\in\mathcal{I}$ such that $\widetilde{I}\subseteq I$.

For every  $\ee>0$, $D\subseteq [0,1]$ and $H_1,...,H_k$ in
$V_2^0$ we will say that $D$ $\ee-$\textit{determines the
quadratic variation of the linear span} $<H_1,...,H_k>$, if for
every $\mathcal{I}\in\mathcal{A}$ there is
$\widetilde{\mathcal{I}}\preceq \mathcal{I}$ in $\mathcal{F}(D)$
such that
 \[\Big|\textit{v}_2^2\Big(\sum_{i=1}^k\lambda_i H_i,\mathcal{I}\Big)-
\textit{v}_2^2\Big(\sum_{i=1}^k\lambda_i
H_i,\widetilde{\mathcal{I}}\Big)\Big|\leq
\Big(\sum_{i=1}^k|\lambda_i|^2\Big)\ee,\]  for every sequence of
scalars $(\lambda_i)_{i=1}^k$. Using standard approximation
arguments the following is easily proved.
\begin{prop}\label{det}
Let $k\in\nn$, $H_1,\ldots,H_k$ in $V_2^0$ and $\ee>0$. Then there
exists $\delta>0$ such that every $D\subseteq [0,1]$ which is
$\delta-$dense in $[0,1]$,
 $\ee-$determines the quadratic variation of  $<H_1,\ldots,H_k>$.
\end{prop}

Next we state some notation for the dyadic tree. For every $n\geq
0$, we set $2^n=\{0,1\}^n$ (where $2^0=\{\emptyset\}$). Hence for
$n\geq 1$, every $s\in 2^n$ is of the form $s=(s(1),...,s(n))$.
For $0\leq m<n$ and $s\in 2^n$, $s|m=(s(1),...,s(m))$, where if
$m=0$, $s|0=\emptyset$. Also, $2^{\leqslant n}=\cup_{i=0}^ n 2^i$
and $2^{<\nn}=\cup_{n=0}^\infty 2^n$. The \textit{length} $|s|$ of
an $s\in 2^{<\nn}$, is the unique $n\geq 0$ such that $s\in 2^n$.
The initial segment partial ordering on $2^{<\nn}$ will be denoted
by $\sqsubseteq$ (i.e. $s\sqsubseteq t$ if $m=|s|\leq |t|$ and
$s=t|m$). For $s, t\in 2^{<\nn}$, $s\perp t$ means that $s,t$ are
$\sqsubseteq$-incomparable (that is neither $s\sqsubseteq t$ nor
$t\sqsubseteq s$). For an $s\in 2^{<\nn}$, $s^\smallfrown 0$ and
$s^\smallfrown 1$ denote the two immediate successors of $s$ which
end with $0$ and $1$ respectively. More generally for $s,u\in
2^{<\nn}$, $s^\smallfrown u$ denotes the concatenation of $s$ and
$u$, namely the element $t\in 2^{<\nn}$ with $|t|=|s|+|u|$,
$t(i)=s(i)$ for all $1\leq i\leq |s|$ and $t(|s|+i)=u(i)$ for all
$1\leq i\leq |u|$.

An \textit{antichain} of $2^{<\nn}$, is a subset of $2^{<\nn}$
such that for every $s,t\in A$, $s\perp t$. A \textit{branch} of
$2^{< \nn}$ is a maximal totally ordered subset of $2^{< \nn}$.  A
\textit{dyadic subtree} is a subset $T$ of $2^{<\nn}$ such that
there is an order isomorphism $\phi:2^{<\nn}\to T$. In this case
$T$ is denoted by $T=(t_s)_{s\in 2^{<\nn}}$, where $t_s=\phi(s)$.

In the sequel by the term \textit{subspace} we always mean closed
infinite dimensional subspace. We also use the standard notation
for Banach spaces from \cite{LT}.

\subsection{The discontinuities of $X^{**}$ for subspaces $X$ of $V_2^0$.}
For every  $f\in V_2$, by  $D_f$   we denote the set of all points
of discontinuity of $f$. For all $t\in [0,1]$ let
$f(t^+)=\lim_{s\rightarrow t^+}f(s)$ and $
f(t^-)=\lim_{s\rightarrow t^-}f(s)$ (where by convention we set
$f(0^-)=f(0)$ and $f(1^+)=f(1)$). It is easily shown  that  for
every  $f\in V_2$, the set $D_f$  is at most countable an so $f$
is a Baire-1 function. Moreover for every $t\in D_f$, $f(t^-)$ and
$f(t^+)$ always exist and  $ \sum_{t\in
D_f}|f(t)-f(t^-)|^2+|f(t)-f(t^+)|^2\leq \|f\|^2_{V_2}$.

In this subsection we will study the set  $D_{X^{**}}=\cup_{f\in
X^{**}}D_f$, for subspaces $X$ of $V_2^0$ with $X^{*}$ separable
and $X^{**}$ non-separable. We will show that $D_{X^{**}}$ is a
countable subset of $[0,1]$ which as we will see implies that the
space $(X^{**}, \|\cdot\|_\infty)$ is separable.  We start with a
characterization of the subspaces $X$ of $V_2^0$ with separable
dual through the discontinuity points of all $f\in X^{**}$.
\begin{prop}\label{pcd}
Let $X$ be a subspace of $V_2^0$. Then $X^{*}$ is separable if and
only if $D_{X^{**}}$ is countable.
\end{prop}
\begin{proof}
Suppose  that  $D_{X^{**}}$ is uncountable. Then, since for every
$f\in X^{**}$, $D_f$ is countable, we may choose uncountable sets
$\mathcal{F}=\{f_\xi\}_{\xi<\omega_1}\subseteq B_{X^{**}}$ and
$A=\{t_\xi\}_{\xi<\omega_1}\subseteq [0,1]$ such that the
following are satisfied.
\begin{enumerate}
\item For every $\xi<\omega_1$, $f_\xi$ is discontinuous at
$t_\xi$. \item Exactly one of the following hold.
\begin{enumerate}
\item [(2a)] For all $\xi<\omega_1$, $f_\xi(t_\xi^+)\neq
f_\xi(t_\xi)$. \item [(2b)] For all $\xi<\omega_1$,
$f_\xi(t_\xi^-)\neq f_\xi(t_\xi)$.
\end{enumerate}
\end{enumerate}
Suppose that  (2a) holds (the other case is similar). Passing to
an uncountable  subset of $\mathcal{F}$ we may assume that there
exists $\dd>0$ such that $|f_\xi(t_\xi^+)-f_\xi(t_\xi)|>\dd$, for
every $\xi<\omega_1$. Moreover, by passing to a further
uncountable subset, we can suppose that there exist $0<\ee<\dd$
and an open interval $I$ of $(0,1)$ such that for every
$\xi<\omega_1$ we have that (i) $t_\xi\in I$, (ii) for every $t\in
I$ and $t<t_\xi$, $|f_\xi(t)-f_\xi(t_\xi^-)|<\ee$ and (iii) for
every $t\in I$ and $t>t_\xi$, $|f_\xi(t)-f_\xi(t_\xi^+)|<\ee$.

Let $\xi<\xi^{'}$.  If $t_\xi<t_{\xi^{'}}$ then we have that
\[|\dd_{t_\xi}(f_\xi)-\dd_{t_{\xi^{'}}}(f_\xi)|\geq |f_\xi(t_\xi)-f_\xi(t_\xi^+)|-
|f_\xi(t_\xi^+)-f_\xi(t_{\xi^{'}})|>\dd-\ee.\] and  if
$t_{\xi^{'}}<t_\xi$ then similarly $|\dd_{t_\xi}(f_{\xi^{'}})-\dd_{t_{\xi^{'}}}(f_{\xi^{'}})|>\dd-\ee$.\\
This implies that
$\|\dd_{t_\xi}|_{X}-\dd_{t_{\xi^{'}}}|_{X}\|\geq\dd-\ee$ for every
$\xi\neq\xi^{'}$ and therefore $X^*$ is nonseparable. Finally for
the converse, suppose that  $X^*$ is non-separable. Then by
Proposition 23 of \cite{AAK} we have that $X^{**}$ contains a non
separable  family $\mathcal{H}\subseteq V_2^d=\overline
{<\{\chi_t:t\in(0,1)\}>}$ and therefore $D_{X^{**}}$ must be
uncountable.
\end{proof}

\begin{prop}\label{ps}
Let $\mathcal{F}$ be a subset of $V_2$ such that
$D_\mathcal{F}=\bigcup_{f\in\mathcal{F}} D_f$ is countable. Then
the  space $(\mathcal{F},\|\cdot\|_\infty)$ is separable.
\end{prop}
\begin{proof} We set
$Y=V_2^c+<\{\chi_{[t,1]}:t\in D_{\mathcal{F}}\cap (0,1]\}>$ (where
$V_2^c=V_2\cap C[0,1]$).  By exploiting the proof of Theorem 15 of
\cite{AAK}, we have that
$\mathcal{F}\subseteq\overline{Y}^{\|\cdot\|_{V_2}}\oplus\ell_2(D_\mathcal{F})$.
As $\|\cdot\|_{\infty}\leq \|\cdot\|_{V_2}$, we get that
$\overline{Y}^{\|\cdot\|_{V_2}}\subseteq
\overline{Y}^{\|\cdot\|_{\infty}}$. Moreover since $D_\mathcal{F}$
is countable, $(\ell_2(D_{\mathcal{F}}), \|\cdot\|_{\infty})$ is
separable. Since $(Y, \|\cdot\|_\infty)$ is also separable, the
result follows.\end{proof}

\begin{cor}\label{corps} Let $X$ be a subspace of $V_2^0$ such
that $X^*$ is separable. Then the space
$(X^{**},\|\cdot\|_\infty)$ is separable.
\end{cor}

\subsection{Biorthogonal families in $V_2^0$.}
\subsubsection{Definition and existence.}
In this subsection we introduce the concept of biorthogonality
for families of functions in $V_2^0$.
\begin{defn}\label{B}
Let $(H_i)_{i\in S}$ be a family of functions of $V_2^0$ and $(\varepsilon_i)_{i\in S}$
a family of positive real numbers, where $S$ is a
countable set. We will say that $(H_i)_{i\in S}$ is
$(\varepsilon_i)_{i\in S}-$ biorthogonal, if for every $\mathcal{I}\in \mathcal{A}$
there is a disjoint partition $\mathcal{I}=\cup_{i\in S}\mathcal{I}^{(i)}$ such that for every $j\in S$,
\begin{equation}\label{biorth.}\sum_{\{i\in S: i\neq j\}}\textit{v}_2(H_i,\mathcal{I}^{(j)})\leq\varepsilon_j \end{equation}
\end{defn}
\begin{prop}\label{pb}
Let $(H_n)_{n\in\nn}$ be a sequence of functions of $V_2^0$ with
$\lim \|H_n\|_{\infty}= 0$. Then for every sequence
$(\ee_i)_{i\in\nn}$ of positive real numbers there exists a
subsequence $(H_{n_i})_{i\in\nn}$ such that $(H_{n_i})_{i\in\nn}$
is $(\ee_i)_{i\in\nn}-$ biorthogonal.
\end{prop}

For the proof of the above proposition we will need  some specialized forms of biorthogonality.
 Let $k\geq 1$, $(\ee_i)_{i=1}^k$ and $(\delta_i)_{i=0}^{k-1}$ be  finite sequences of positive real numbers
 such that
 $0<\delta_{k-1}<...<\delta_1<\delta_0=1$.
We say that a sequence $(H_i)_{i=1}^k$ in
 $ V_2^0$ is $((\ee_i)_{i=1}^k, (\delta_i)_{i=0}^{k-1})-$
biorthogonal if the inequality (\ref{biorth.}) of Definition
\ref{B} is satisfied for $S=\{1,...,k\}$ and
\begin{equation}\label{mal1}\mathcal{I}^{(j)}=\{I\in\mathcal{I}:
\delta_j<|I|\leq\delta_{j-1}\},\end{equation}  for all $1\leq
j\leq k$ (where for $j=k$, we set  $\delta_k=0$).

 We will also
use the following notation. For  a sequence of positive real
numbers $(\ee_i)_{i\in\nn}$ and every $1\leq i\leq k$,  we set
$\ee_i^k=(\sum_{r=1}^{k-i+1}2^{-r})\ee_i$. Clearly for all $1\leq
i\leq k$, $\ee_i^k<\ee_i^{k+1}$ and $\lim_k \ee_i^k=\ee_i$.

The proof of Proposition \ref{pb}  is based on  the next lemma.

\begin{lem}\label{bio}
 Let
$(H_i)_{i=1}^k$ be an $((\ee^k_i)_{i=1}^k,
(\delta_i)_{i=0}^{k-1})-$ biorthogonal sequence of $V_2^0$. Let
$0<\delta_k<\delta_{k-1}$ be such that for every
$\mathcal{I}\in\mathcal{A}$ with
$\|\mathcal{I}\|_{\max}\leq\delta_k$, the following holds.
\begin{equation}\label{mal}\sum_{i=1}^{k}\textit{v}_2(H_i,\mathcal{I})\leq\frac{\ee_{k+1}}{2}\end{equation}
 Then there is  $\ee>0$ such
that  for every $H_{k+1}\in V_2^0$ with
$\|H_{k+1}\|_{\infty}<\epsilon$, the sequence $(H_i)_{i=1}^{k+1}$
is $((\ee_i^{k+1})_{i=1}^{k+1}, (\delta_i)_{i=0}^{k})-$
biorthogonal.
\end{lem}

\begin{proof} Notice that for every  $\mathcal{J}\in\mathcal{A}$ with
$\delta_{k}<\|\mathcal{J}\|_{\min}$,
 $card(\mathcal{J})<\delta_k^{-1}$. Set
 $\epsilon=\sqrt{\dd_k}2^{-(k+3)}\min\{\ee_i\}_{i=1}^k$ and
let  $H_{k+1}\in V_2^0$ be such that
$\|H_{k+1}\|_{\infty}<\varepsilon$. Then for every  $1\leq j\leq
k$ and $\mathcal{I}\in\mathcal{A}$, we have that
$\dd_k<\dd_j<\|\mathcal{I}^{(j)}\|_{\min}$ and therefore
\[\textit{v}_2(H_{k+1},\mathcal{I}^{(j)})\leq
\Big(\delta_k^{-1}(2\|H_{k+1}\|_{\infty})^2\Big)^{1/2}\leq
\frac{\ee}{2^{k+2}}\leq \frac{\ee_j}{2^{k-j+2}}\]  Hence for each
$1\leq j\leq k$ and for every $\mathcal{I}\in \mathcal{A}$,
\[\begin{split}\sum_{\{i:1\leq i\leq k+1, i\neq j\}}\textit{v}_2(H_i, \mathcal{I}^{(j)})
&=\sum_{\{i:1\leq i\leq k, i\neq j\}}\textit{v}_2(H_i,
\mathcal{I}^{(j)})+ \textit{v}_2(H_{k+1},\mathcal{I}^{(j)})\\&\leq
\ee_j^k+\frac{\ee_j}{2^{k-j+2}}=
\Big(\sum_{r=1}^{k-j+2}2^{-r}\Big)\ee_j=\ee_j^{k+1}\end{split}\]
Finally  $\|\mathcal{I}^{(k+1)}\|_{\max}\leq\delta_{k}$ and so by
(\ref{mal}), we get that  \[\sum_{i=1}^k\textit{v}_2(H_i,
\mathcal{I}^{(k+1)})\leq\frac{\ee_{k+1}}{2}=\ee_{k+1}^{k+1}\]\end{proof}

\noindent\textbf{Proof of Proposition \ref{pb}.} We inductively
construct an increasing sequence $n_1<n_2<...$ of natural numbers
and a decreasing sequence of positive real numbers
$0<...<\delta_2<\delta_1<1=\delta_0$, such that for every $k\geq
1$, the sequence $(H_{n_i})_{i=1}^k$ is
$((\ee_i^k)_{i=1}^k,(\delta_i)_{i=1}^{k-1})$-biorthogonal. We
claim that $(H_{n_i})_i$ is $(\ee_i)_i$-biorthogonal. Indeed, let
$\mathcal{I}\in\mathcal{A}$ and let
$\mathcal{I}=\cup_i\mathcal{I}^{(i)}$ be the partition of
$\mathcal{I}$ induced by (\ref{mal1}). Let also $k_0\geq 1$ be
such that $\delta_{k_0}<\|\mathcal{I}\|_{\min}$.  Then for each
$j\in\nn$ with  $j\geq k_0$, $\mathcal{I}^{(j)}=\emptyset$ and so
(\ref{biorth.}) trivially holds. Otherwise for all $k\geq 1$,
$\sum_{\{1\leq i\leq k:\; i\neq
j\}}\textit{v}_2(H_{n_i},\mathcal{I}^{(j)})<\ee^k_j$ and so
$\sum_{\{ i\neq j\}}\textit{v}_2(H_{n_i},\mathcal{I}^{(j)})\leq
\ee_j.$

\bigskip

We will also need the analogue of the above  in the case where
$S=2^{<\nn}$. We omit the proof since it is an easy modification
of the one  of Proposition \ref{pb}.

\begin{prop}\label{treebio}
Let $(H_s)_{s\in 2^{<\nn}}$ be a family of functions in $V_2^0$
such that for every $\sigma\in 2^{\nn}$,
 $\lim_n\|H_{\sigma\mid n}\|_{\infty}= 0$. Then for every family $(\ee_s)_{s\in 2^{<\nn}}$ of positive real numbers,
there exists a dyadic subtree $(t_s)_{s\in 2^{<\nn}}$ of
$2^{<\nn}$ such that $(H_{t_s})_{s\in 2^{<\nn}}$ is $(\ee_s)_{s\in
2^{<\nn}}-$ biorthogonal.
\end{prop}

\subsubsection{Estimations on biorthogonal sequences.}
In the next two lemmas and proposition,  $S$ stands for  a
countable set and $(H_i)_{i\in S}$ is an $(\varepsilon_i)_{i\in
S}-$biorthogonal family in $V_2^0$ such that $\sum_{i\in
S}\ee_i=\ee<\infty.$
\begin{lem}\label{lprep} Let   $\mathcal{I}\in
\mathcal{A}$, $F\subseteq S$ finite and $(\lambda_i)_{i\in F}$ be
a  sequence of real numbers. Then
\begin{enumerate}
\item [(i)] For every $i\in S$,
$\textit{v}_2(H_i,\mathcal{I}\setminus \mathcal{I}^{(i)})<\ee.$
\item[(ii)] For every $j\notin F$, $\textit{v}_2(\sum_{i\in
F}\lambda_i
 H_i,\mathcal{I}^{(j)})
 \leq \max_{i\in F}|\lambda_i
 |\ee_j$.
 \item [(iii)] For every $j\in
F$, $\textit{v}_2(\sum_{i\in F}\lambda_i
 H_i,\mathcal{I}^{(j)})
 \leq |\lambda_j|\textit{v}_2(H_j,\mathcal{I}^{(j)})+\max_{i\in F}|\lambda_i
 |\ee_j$.
 \item [(iv)] For every $j\in F$,
$\textit{v}_2(\sum_{i\in F}\lambda_i
 H_i,\mathcal{I}^{(j)})\geq\Big||\lambda_j|\textit{v}_2(H_j,\mathcal{I}^{(j)})-\max_{i\in
F}|\lambda_i
 |\ee_j\Big|
 $.
\end{enumerate}
\end{lem}
\begin{proof} (i) Let $i\in S$. Then \[\textit{v}_2(H_i,\mathcal{I}\setminus
\mathcal{I}^{(i)}) =\textit{v}_2(H_i,\cup_{j\neq
i}\mathcal{I}^{(j)}) \leq \sum_{j\neq
i}\textit{v}_2(H_i,\mathcal{I}^{(j)})\leq \sum_{j\neq
i}\ee_j<\ee\] (ii) Let $j\notin F$. Then
\[\textit{v}_2(\sum_{i\in F}\lambda_i
 H_i,\mathcal{I}^{(j)})
 \leq\sum_{i\in F}\textit{v}_2(\lambda_i H_i,\mathcal{I}^{(j)})=
\sum_{i\in F}|\lambda_i|\textit{v}_2( H_i,\mathcal{I}^{(j)})\leq
\max_{i\in F}|\lambda_i|\ee_j\] (iii) Let $j\in F$. Then using
(ii) we get that
\[\begin{split}\textit{v}_2(\sum_{i\in F}\lambda_i
 H_i,\mathcal{I}^{(j)})& \leq
\textit{v}_2(\lambda_j H_j,\mathcal{I}^{(j)})+\sum_{i\neq
j}\textit{v}_2(\lambda_i H_i,\mathcal{I}^{(j)})\\&\leq
|\lambda_j|\textit{v}_2( H_j,\mathcal{I}^{(j)})+\max_{i\in
F}|\lambda_i|\ee_j
\end{split}\]
(iv) Since $ \textit{v}_2(\sum_{i\in F}\lambda_i
 H_i,\mathcal{I}^{(j)})\geq |\textit{v}_2(\lambda_j H_j,\mathcal{I}^{(j)})-\sum_{i\neq
j}\textit{v}_2(\lambda_i H_i,\mathcal{I}^{(j)})|$, the proof is
similar to that of (ii).
\end{proof}
\begin{lem}\label{lul2}
Let $M>0$ and suppose that  $\|H_i\|_{V_2}\leq M$, for all $i\in
S$. Then for all finite subsets $F\subseteq G\subseteq  S,$ every
sequence of scalars $(\lambda_i)_{i\in G}$ and  every
$\mathcal{I}\in\mathcal{A}$ the following are satisfied.
\begin{enumerate}
\item [(i)] $\textit{v}_2^2\Big(\sum_{i\in F}\lambda_i
H_i,\mathcal{I}\Big)\leq\sum_{i\in
F}|\lambda_i|^2\textit{v}_2^2(H_i,\mathcal{I}^{(i)}) +\max_{i\in
F}|\lambda_i|^2(2M+\ee)\ee$. \item
[(ii)]$\textit{v}_2^2\Big(\sum_{i\in
G}\lambda_iH_i,\mathcal{I}\Big)> \sum_{i\in
F}|\lambda_i|^2\textit{v}_2^2(H_i,\mathcal{I}^{(i)})- \max_{i\in
G}|\lambda_i|^22M\ee$. \item [(iii)]
$\textit{v}_2^2\Big(\sum_{i\in F}\lambda_i
H_i,\mathcal{I}\big)\leq\textit{v}_2^2\Big(\sum_{i\in G}\lambda_i
H_i,\mathcal{I}\Big)+ \max_{i\in G}|\lambda_i|^2(4M+\ee)\ee$.
\end{enumerate}
\end{lem}
\begin{proof} (i) By (ii) and (iii) of Lemma \ref{lprep}, we have
that
\[\begin{split}  \;\;\textit{v}_2^2\Big(\sum_{i\in F}\lambda_i H_i,\mathcal{I}\Big)
&= \sum_{j\in F}\textit{v}_2^2\Big(\sum_{i\in F}\lambda_i
H_i,\mathcal{I}^{(j)}\Big) +\sum_{j\in S\setminus
F}\textit{v}_2^2\Big(\sum_{i\in F}\lambda_i
H_i,\mathcal{I}^{(j)}\Big)\\& \leq\sum_{j\in
F}\Big(|\lambda_j|\textit{v}_2( H_j,\mathcal{I}^{(j)})+\max_{i\in
F}|\lambda_i|\ee_j \Big)^2+\sum_{j\in S\setminus F}\max_{i\in
F}|\lambda_i|^2\ee_j^2
\\&
\leq\sum_{j\in
F}|\lambda_j|^2\textit{v}_2^2(H_j,\mathcal{I}^{(j)}) +\max_{j\in
F}|\lambda_j|^2(2M+\ee)\ee.\end{split}\] (ii) Using (iv) of Lemma
\ref{lprep}, we obtain that
\[\begin{split}
&\textit{v}_2^2\Big(\sum_{i\in
G}\lambda_iH_i,\mathcal{I}\Big)\geq\textit{v}_2^2(\sum_{i\in
G}\lambda_i H_i,\cup_{j\in F}\mathcal{I}^{(j)})= \sum_{j\in
F}\textit{v}_2^2(\sum_{i\in G}\lambda_i H_i,\mathcal{I}^{(j)})
\\&\geq\sum_{j\in
F}\Big||\lambda_j|\textit{v}_2(H_j,\mathcal{I}^{(j)})-\max_{i\in
G}|\lambda_i
 |\ee_j\Big|^2\geq\sum_{j\in F}|\lambda_j|^2\textit{v}_2^2(H_j,\mathcal{I}^{(j)})-
\max_{i\in G}|\lambda_i|^22M\ee.
\end{split}\]
Finally (iii) follows easily from (i) and (ii).\end{proof}
\begin{prop}\label{punc} 
 Let
$M>\theta>2\ee>0$ and suppose that $\theta<\|H_i\|_{V_2}\leq M$,
for all $i\in S$. Then $(H_i)_{i\in S}$ is an unconditional
family.\end{prop}
\begin{proof}
Let $F\subseteq G$ be finite subsets of $S$ and let
$|\lambda_{i_0}|=\max_{i\in G}|\lambda_i|$. By (iii) of Lemma
\ref{lul2},  we easily get that
\begin{equation}\label{punc1}
\Big\|\sum_{i\in F}\lambda_iH_i\Big\|^2_{V_2}\leq \Big\|\sum_{i\in
G}\lambda_iH_i\Big\|^2_{V_2}+|\lambda_{i_0}|^2(4M+\ee)\ee,
\end{equation}
Let $\mathcal{I}_0\in\mathcal{A}$ such that
$\textit{v}_2(H_{i_0},\mathcal{I}_0)>\theta$. Since $(H_i)_{i\in
S}$ is $(\ee_i)_{i\in S}-$biorthogonal, we get that
\begin{equation}\label{punc2}\textit{v}_2(H_{i_0},\mathcal{I}_0^{(i_0)})\geq\textit{v}_2(H_{i_0},\mathcal{I}_0)-
\textit{v}_2(H_{i_0},\mathcal{I}\setminus\mathcal{I}_0^{(i_0)})>\theta-\ee\end{equation}
Moreover by (ii) of Lemma \ref{lprep}, we have that
\[\begin{split}
\Big|\textit{v}_2(\lambda_{i_0}H_{i_0},\mathcal{I}_0^{(i_0)})-
\textit{v}_2\Big(\sum_{i\in
G}\lambda_iH_i,\mathcal{I}_0^{(i_0)}\Big)\Big|&\leq
\textit{v}_2\Big(\sum_{i\in G,i\neq i_0}\lambda_i
H_i,\mathcal{I}_0^{(i_0)}\Big)\leq|\lambda_{i_0}|\ee,
\end{split}\] and so
\begin{equation}\label{uncbio}|\lambda_{i_0}|\textit{v}_2(H_{i_0},\mathcal{I}_0^{(i_0)})
\leq\textit{v}_2\Big(\sum_{i\in
G}\lambda_iH_i,\mathcal{I}_0^{(i_0)}\Big)+|\lambda_{i_0}|\ee\leq
\Big\|\sum_{i\in G}\lambda_iH_i\Big\|_{V_2}+|\lambda_{i_0}|\ee
\end{equation}
By  (\ref{punc2}) and (\ref{uncbio}), we get that
\[|\lambda_{i_0}|\leq \frac{1}{\theta-2\ee}\Big\|\sum_{i\in
G}\lambda_iH_i\Big\|_{V_2}\] Hence by (\ref{punc1}), we have that
\[\Big\|\sum_{i\in F}\lambda_iH_i\Big\|_{V_2}\leq \Big(1+\frac{(4M+\ee)\ee}{(\theta-2\ee)^2} \Big)^{1/2}\Big\|\sum_{i\in
G}\lambda_iH_i\Big\|_{V_2}\] and the proof of the proposition is
complete.
\end{proof}

\section{ Hausdorff   measures  associated to  functions of  bounded quadratic variation.}
The aim of  this section is to  introduce and study the
fundamental properties of the measure $\mu_f$ corresponding to a
function  $f\in V_2$. It is divided into three subsections. The
first includes the  definition and initial  properties of the
measure $\mu_f$. The second is mainly devoted to the proof of
Theorem \ref{tm} and the last contains a study of the points of
non differentiability of a function $f\in V_2$.

\subsection{Definition and elementary properties.}
 For every  $f\in V_2$ and  for every
interval $I$ of $[0,1]$ we set
\[\widetilde{\mu}_f(I)=\inf_{\delta>0}\widetilde{\mu}_{f,\delta}(I)\]
where for each $\delta>0$,
$\widetilde{\mu}_{f,\delta}(I)=\sup\{\textit{v}_2^2(f,\mathcal{P}):
\mathcal{P}\subseteq I
\;\text{and}\;\|\mathcal{P}\|_{\max}<\dd\}$.

We also  define the function $\widetilde{F}_f:\rr\rightarrow\rr$
by
\begin{equation*}
\widetilde{F}_f(x)=
\begin{cases}
0,   & \text{if}\,\,\,x\leq 0\\
\widetilde{\mu}_f[0,x], &  \text{if}\,\,\, 0<x<1\\
\widetilde{\mu}_f[0,1],  & \text{if}\,\,\, x\geq 1
\end{cases}
\end{equation*}

Notice  that $\widetilde{F}_f$ is a non-negative increasing
function on $\rr$    and so taking the upper envelope
$F_f(x)=\widetilde{F}_f(x^+)$ of $\widetilde{F}_f$,  we have that
$F_f$ is in addition a right continuous function. Moreover since
$\lim_{x\to-\infty} F_f(x)=0$ and $\lim_{x\to+\infty}
F_f(x)=\widetilde{\mu}_f[0,1]$, $F_f$ is the distribution function
of a finite positive Borel measure on $\rr$ which we will denote
by $\mu_f$. Notice   that $\mu_f=0$ if and only if $f\in V_2^0$
and  also that $\mu_f\leq \|f\|_{V_2}$. Actually, it is easy to
see that  defining for any function $f:[0,1]\to\rr$, the measure
$\mu_f$ as above, then $\mu_f$ is finite if and only if $f\in
V_2$. Since $\mu_f(-\infty,0)=\mu_f(1,+\infty)=0,$ in the sequel
we will identify $\mu_f$ with its restriction on $[0,1]$.

\begin{rem} The definition of the measure $\mu_f$ is generalized
as follows. Let $(X,\rho)$ be a metric space and $f:X\to \rr$ be a
real valued function. Following C.A. Rogers in \cite{Rog}, for a
function $h:[0,+\infty]\to[0,+\infty]$ satisfying the conditions
of p.50 of \cite{Rog} and every open subset $G$ of $X$, we define
the premeasure $h_f(G)=h(diam f[G])$. Next following Method II
(see \cite{Rog}), we induce the measure $\mu_f^h$ defined on the
Borel subsets of $X$. It is easy to see that in the case of $f\in
V_2$ and for $h(x)=x^2$, the measure $\mu_f^h$ coincides with the
measure $\mu_f$ defined above. Although the measures $\mu_f^h$ are
not mentioned as Hausdorff measures in  the literature, their
definition and geometrical properties motivate us  to include them
in the latter class. It seems interesting to examine the
regularity conditions that a function $f:X\to\rr$ must satisfy so
that the corresponding measure $\mu_f^h$ is a finite Borel
measure. For example, this easily yields that $f$ has at most
countably many discontinuities. Hence if $X$ is a Polish space $f$
is a Baire-1 function.
\end{rem}
The following two lemmas are easily proved. The second one is
essentially contained in \cite{AAK} (Lemma 18).

\begin{lem} \label{simple1}Let  $0\leq a<x<b\leq 1$, such that $f$ is continuous at $x$.
 Then for every $I=I_1\cup I_2$, where $I$ is an interval with  endpoints $a,b$ and
  $I_1, I_2$ are intervals with $\sup I_1=x= \inf
 I_2$, we have that
\[\widetilde{\mu}_f(I)=\widetilde{\mu}_f(I_1)+
\widetilde{\mu}_f(I_2)\]
\end{lem}

\begin{lem}\label{simple2}
\text{(a)} For every  $x\in(0,1]$ and  every $\ee>0$ there exists
$0<\dd<x$ such that $\sup\{\textit{v}_2^2(f,\mathcal{P})
:\mathcal{P}\subseteq [x-\delta,x)\}\leq\ee.$ In particular
$\widetilde{\mu}_f[x-\delta,x)\leq\ee.$

(b) Similarly for every $x\in[0,1)$  and  every $\ee>0$ there
exists $0<\delta<1-x$ such that
$\sup\{\textit{v}_2^2(f,\mathcal{P}) :\mathcal{P}\subseteq (x,
x+\delta]\leq\ee.$ In particular
$\widetilde{\mu}_f(x, x+\delta]\leq\ee.$ \end{lem}

\begin{prop}\label{pm}  \begin{enumerate} \item[(a)] $D_f=D_{F_f}=D_{\widetilde{F}_f}$ and
so $f$ is continuous if and only if $\mu_f$ is continuous. Also  for $x\in [0,1]\setminus D_f$,
$\widetilde{F}_f(x)=F_f(x)$.
\item [(b)] If $f$ is continuous at $x$ then $\mu_f[0,x]=\widetilde{\mu}_f[0,x]$ and
$\mu_f[x,1]=\widetilde{\mu}_f[x,1].$
\item [(c)]  For all continuity points  $x<y$ of $f$, $\mu_f[x,y]=\widetilde{\mu}_f[x,y]$.
\item [(d)]  For every
open interval $(\alpha,\beta)$ of $[0,1]$, $\mu_f(\alpha,\beta)=\widetilde{\mu}_f(\alpha,\beta)$.
\end{enumerate}
\end{prop}
\begin{proof} (a) By the monotonicity of $\widetilde{F}_f$ we have that for all $x_0\in\rr$,
$\widetilde{F}_f(x_0^+)=F_f(x_0^+)$ and $\widetilde{F}_f(x_0^-)=F_f(x_0^-)$.
Therefore $D_{F_f}=D_{\widetilde{F}_f}$ and  for every $x_0\in [0,1]\setminus D_{\widetilde{F}_{f}}$,
 $\widetilde{F}_f(x_0)=F_f(x_0)$. It remains to show that $D_f=D_{\widetilde{F}_{f}}$.
Let $x_0\in[0,1]$ be a continuity point of $f$. If $x_0<1$, by Lemma \ref{simple1}  we have that
 for every $x_0<y\leq 1$,
  $\widetilde{F}_f(y)-\widetilde{F}_f(x_0)=\widetilde{\mu}_f[0,y]-\widetilde{\mu}_f[0,x_0]=\widetilde{\mu}_f(x_0, y]$
 and so  by Lemma \ref{simple2}(b),
$\widetilde{F}_f(x_0^+)=\widetilde{F}_f(x_0)$. If $0<x_0$,  again
by Lemma \ref{simple1},
 for every $0\leq y<x_0$ such that $y$ is a continuity point of $f$ we have that
$\widetilde{F}_f(x_0)-\widetilde{F}_f(y)=\widetilde{\mu}_f(y,x_0]$. Since $[0,1]\setminus D_f$ is dense in $[0,1]$,
 by Lemma \ref{simple2}(b)
we get  that $\widetilde{F}_f$ is continuous at $x_0$. Conversely suppose that $x_0\in D_f$.
 Then either $f(x_0^+)\neq f(x_0)$ or $f(x_0^-)\neq f(x_0)$.
  Suppose that $f(x_0^+)\neq f(x_0)$ (the other case is similarly treated). Then it is easy to see that
for every $0<\delta<1-x_0$,  $\widetilde{\mu}_{f}[x_0,x_0+\delta]\geq |f(x_0^+)- f(x_0)|^2$
and using   Lemma \ref{simple1}, we obtain that
\[\begin{split}\widetilde{F}_f(x_0^+)=\lim_{\dd\to 0}\widetilde{F}_f(x_0+\dd)&\geq\lim_{\dd\to 0}(\widetilde{\mu}_f[0,x_0]+\widetilde{\mu}_f[x_0,x_0+\dd])\\&\geq
 \widetilde{F}_f(x_0)+|f(x_0^+)- f(x_0)|^2>\widetilde{F}_f(x_0)\end{split}\] Hence $x_0\in D_{\widetilde{F}_f}$.\\
(b) If $f$ is continuous at $x$ then by  (a) we have that  $\widetilde{F}_f(x)=F_f(x)$
or  $\mu_f[0,x]=\widetilde{\mu}_f[0,x]$. Again by (a) we have that $F_f$ is continuous at $x$
and so $\mu_f(\{x\})=0$. Hence $\mu_f[x,1]=\mu_f[0,1]-\mu_f[0,x)=\mu_f[0,1]-\mu_f[0,x]=\widetilde{\mu}_f[0,1]-\widetilde{\mu}_f[0,x]=
\widetilde{\mu}_f[x,1]$, where the last equality follows from Lemma \ref{simple1}.\\
(c) Indeed, using (b) and Lemma \ref{simple1}, $\mu_f[x,y]=\mu_f[0,y]-\mu_f[0,x)=
\mu_f[0,y]-\mu_f[0,x]=\widetilde{\mu}_f[0,y]-\widetilde{\mu}_f[0,x]=\widetilde{\mu}_f[x,y]$.\\
(d) Let $a<a_n<b_n<b$ such that $a_n,b_n$ are continuity points of $f$, $\lim_n a_n=a$ and $\lim_n b_n=b$.
By Lemma \ref{simple2}, $\lim_n\widetilde{\mu}_f(a,a_n]=\lim_n\widetilde{\mu}_f[b_n,b)=0$ and
 by Lemma \ref{simple1}, $\widetilde{\mu}_f(a,b)=\widetilde{\mu}_f(a,a_n]+\widetilde{\mu}_f[a_n,b_n]+\widetilde{\mu}_f[b_n,b)$.
Hence
$\widetilde{\mu}_f(a,b)=\lim_n\widetilde{\mu}_f[a_n,b_n]=\lim_n\mu_f[a_n,b_n]=
\mu_f(a,b)$. \end{proof}

For the following we need the next notation. For every $f\in V_2$
and $x_0\in [0,1]$, let
\[\tau_f(x_0)=\max\{|f(x_0^+)-f(x_0^-)|^2,
|f(x_0^+)-f(x_0)|^2+|f(x_0^-)-f(x_0)|^2\}\] where $f(0^-)=f(0)$
and $f(1^+)=f(1)$. Moreover,  for every  $\delta>0$, let
 \[\tau_{f,\dd}(x_0)=\sup\{|f(y)-f(x_0)|^2+|f(x_0)-f(z)|^2, |f(y)-f(z)|^2\},\] where the $\max$ is taken for all $0\leq y\leq
 x_0\leq z\leq 1$ with $|y-z|\leq
\dd$. Clearly $\lim_{\dd\to 0}\tau_{f,\dd}(x_0)=\tau_f(x_0)$ and
$f$ is continuous at $x_0$ if and only if $\tau_f(x_0)=0$.

 The
quantity $\tau_f(x_0)$, is defined (with different notation) in
\cite{P-W}, p.1464.  An equivalent definition was introduced
earlier by L.C. Young \cite{Y}, in order to characterise the class
$\mathcal{W}^*_2$. The proof of the next proposition uses similar
arguments to  those in  \cite{P-W}.
\begin{prop} \label{discr.} For every $f\in V_2$ and $x_0\in [0,1]$,  $\mu_f(\{x_0\})=\tau_f(x_0)$.
Therefore $\mu_f^d [0,1]=\sum_{x\in D_f}\tau_f(x)$, where
$\mu_f^d$ denotes the discrete part of $\mu_f$.
\end{prop}

\begin{rem} Under  the current terminology, it follows that
$\mathcal{W}^*_2=\{f\in V_2:\mu_f=\mu_f^d\}$.\end{rem}

 \begin{lem}\label{la}
Let $f_1,f_2\in V_2$ and $\tau$ be a positive Borel measure on
$[0,1]$  such that $\mu_{f_1}\perp\tau$ and $\tau\leq \mu_{f_2}$.
Then $\tau\leq \mu_{f_2-f_1}$.
\end{lem}
\begin{proof}
Let $F=f_1-f_2$. It suffices to show that $\mu_{F}(V)\geq\tau(V)$,
for every open subset $V$ of $[0,1]$. So fix   an open subset $V$
of $[0,1]$ and let $\ee>0$. Since $\mu_{f_1}\perp\tau$, there is
$V_{\ee}\subseteq V$ such that $V_{\ee}=\cup_{i=1}^kI_i$, where
$(I_i)_{i=1}^k$ is a finite family of pairwise disjoint open
intervals of $[0,1]$, $\mu_{f_1}(V_{\ee})<\ee$ and
$\tau(V_{\ee})>\tau(V)-\ee$. Let  $\delta>0$ be such that
$|\widetilde{\mu}_{f_1,\,
\delta}(V_{\ee})-\mu_{f_1}(V_{\ee})|<\ee,$
$|\widetilde{\mu}_{f_2,\,
\delta}(V_{\ee})-\mu_{f_2}(V_{\ee})|<\ee$  and
$|\widetilde{\mu}_{F,\, \delta}(V_{\ee})-\mu_{F}(V_{\ee})|<\ee$.
Also for all $0\leq i\leq k$, let $\mathcal{P}_i\subseteq I_i$
with $|\mathcal{P}_i|<\delta$ and
$\Big|\sum_{i=1}^k\textit{v}_2^2(f_2,\mathcal{P}_i)-\mu_{f_2}(V_{\ee})\Big|<\ee$.
Hence $\sum_{i=1}^k\textit{v}_2^2(f_1,\mathcal{P}_i)<2\ee$ and
\[\begin{split}
\mu_{F}(V)\geq\mu_{F}(V_{\ee})&>\widetilde{\mu}_{F,\delta}(V_{\ee})-\ee
\geq \sum_{i=1}^k\textit{v}_2^2(F,\mathcal{P}_i)-\ee\\&\geq
\sum_{i=1}^k[\textit{v}_2(f_2,\mathcal{P}_i)-\textit{v}_2(f_1,\mathcal{P}_i)]^2-\ee
\\
&\geq
\mu_{f_2}(V_{\ee})-2\Big(\sum_{i=1}^k\textit{v}_2^2(f_2,\mathcal{P}_i)\Big)^{1/2}
\Big(\sum_{i=1}^k\textit{v}_2^2(f_1,\mathcal{P}_i)\Big)^{1/2}-2\ee\\&\geq
\mu_{f_2}(V_{\ee})-2\|f_2\|_{V_2}\sqrt{2\ee}-2\ee\geq
\tau(V_{\ee})-2\|f_2\|_{V_2}\sqrt{2\ee}-2\ee\\&\geq\tau(V)-2\|f_2\|_{V_2}\sqrt{2\ee}-3\ee.
\end{split}\]
Hence, letting $\ee\rightarrow 0$, we get that
$\mu_{F}(V)\geq\tau(V)$ and the proof is
complete.\end{proof}

\subsection{The correspondence between the  functions of $V_2$ and the measures on the unit interval.}
By $\mathcal{M}[0,1]$ we denote the space of all Borel measures on
$[0,1]$ endowed by the norm $\|\mu\|=\sup\{|\mu(B)|: B\;\text{is a
 Borel subset of }\; [0,1]\}$. The positive cone of $\mathcal{M}[0,1]$
  will be denoted by $\mathcal{M}^+[0,1]$. Recall that for every
  $\mu\in\mathcal{M}^+[0,1]$,
  $\|\mu\|=\mu[0,1]$. In this subsection we study the properties
  of the function $\Phi:V_2\to \mathcal{M}[0,1]$, defined by $
  \Phi(f)=\mu_f$, for all $f\in V_2$. We start with the
 following easily established proposition.
 \begin{prop}\label{basineq} The next hold.
\begin{enumerate} \item[(i)] For every
$f_1,f_2\in V_2$, $\mu_{f_1+f_2}\leq 2\mu_{f_1}+2\mu_{f_2}$.\item
[(ii)] For every $f\in V_2$ and $\lambda\in\rr$,  $\mu_{(\lambda
f)}=\lambda^2\mu_f$.\item [(iii)] For every $f\in V_2$ and every
$g\in V_2^0$, $\mu_{f+g}=\mu_f$. \item [(iv)] The map
$\Phi:V_2\to\mathcal{M}[0,1]$, defined by $\Phi(f)=\mu_f$ is
locally Lipschitz. More precisely
\[\|\mu_{f_1}-\mu_{f_2}\|\leq
(\|f_1\|_{V_2}+\|f_2\|_{V_2})\|f_1-f_2\|_{V_2}\]
\end{enumerate}
 \end{prop}
 \begin{rem}\label{rem1} One could not expect that $\Phi$ is
 a linear map as its range is a subset of the positive cone  $\mathcal{M}^+[0,1]$ of
 $\mathcal{M}[0,1]$ ( next we shall show that $\Phi$ is actually onto
 $\mathcal{M}^+[0,1]$). However there are special cases where the
 additivity of the function $\Phi$ is established. For example it
 can be shown that for every pair $f_1,f_2\in V_2$ with
 $\mu_{f_1}\perp\mu_{f_2}$ we have that
 $\mu_{f_1+f_2}=\mu_{f_1}+\mu_{f_2}$. Finally the map $\Phi$ is
 not w$^*$-w$^*$ continuous. For example let $f\in V_2$ such that
 $f=\sum_ng_n$, where  $(g_n)_n$ is a  sequence in $V_2^0$. Then setting
  $f_n=\sum_{k\geq n}g_k$, we have that $(f_n)_n$ pointwise
  converges to $0$, however by (iii) of Proposition \ref{basineq}, $\mu_{f_n}=\mu_f$, for all $n\in\nn$.
 \end{rem}

\begin{lem}\label{lddd} Let $\mu$ be a finite positive discrete
measure on $[0,1]$. Then there is $h\in V_2^d$ such that
$\mu_h=\mu$.\end{lem}
\begin{proof} Let $S=\{t_n\}_n$ be an enumeration of the support of $\mu$. Then
$\mu^d=\sum_{n}\lambda_n\delta_{t_n}$, where
$\lambda_n=\mu^d(\{t_n\})$. We define
$h=\sum_{n}\sqrt{\lambda_n}\chi_{t_n}$ and let
$h_n=\sum_{k=1}^n\lambda_k\chi_{t_k}$. Then $h\in
V_2^d$, $(h_n)_n$ $\|\cdot\|_{V_2}$-converges to $h$ and so
$(\mu_{h_n})_n$ norm  converges to $\mu_h$ in $\mathcal{M}[0,1]$.
Since $\mu_{h_n}=\sum_{k=1}^n\lambda_k\delta_{t_k}$,
$\mu_h=\sum_n\lambda_n\delta_{t_n}=\mu$.\end{proof}

\begin{thm}\label{tm} For every finite  positive Borel measure $\mu$ on
$[0,1]$ there is $f\in V_2$ such that $\mu=\mu_f$.
\end{thm}
\begin{proof}
Since $\mu=\mu^c+\mu^d$ where $\mu^c$ is the continuous and
$\mu^d$ is the discrete part of $\mu$,  by Lemma \ref{lddd}, it
suffices to find $f\in V_2\cap C[0,1]$ such that $\mu^c=\mu_f$ (it
is then easy to see that $\mu_{f+h}=\mu$, where $h\in V_2^d$
satisfying that $\mu_h=\mu^d$). Hence we suppose for the sequel
that $\mu$ is continuous.

 For an interval $I=[a,b]$ in $[0,1]$
let $F_I:I\to \rr$, defined by $F_I(x)=\mu[a,x]$, for all $x\in
I$. Then $F_I$ is continuous , $F_I(a)=0$ and $F_I(b)=\mu(I)$.
Hence we may choose $\xi_I\in(a,b)$ such that
$F_I(\xi_I)=\mu[a,\xi_I]=\mu(I)/2$. Consider now the function
$G_I:I\to \rr$ defined by $G_I(x)=F_I(x)$, if $a \leq x \leq
\xi_I$ and $G_I(x)=\mu_I(I)-F_I(x)$, if $\xi_I\leq x\leq b$.
Clearly $\|G_I\|_\infty=\mu(I)/2$.

\bigskip

 \noindent\textit{Claim} 1. For every interval $I$ of $[0,1],$ let
$H_I=\sqrt{G_I}$.   \begin{enumerate} \item[(i)] For every $x<y\in
I$,
$|H_I(y)-H_I(x)|^2\leq
\mu (x,y]$. \item[(ii)]  For all intervals  $I_1, I_2$ in $[0,1]$
such that $\max I_1\leq \min I_2$ and  for all $x_1\in I_1$,
$x_2\in I_2$,
$|H_{I_2}(x_2)-H_{I_1}(x_1)|^2\leq \mu(x_1,x_2]$.
\end{enumerate}
\begin{proof} (i) Notice that for $\alpha,\beta>0$,
$|\alpha-\beta|^2\leq |\alpha^2-\beta^2|$. Hence
$|H_I(y)-H_I(x)|^2\leq |G_I(y)-G_I(x)|$ and so it suffices to show
that $|G_I(y)-G_I(x)|\leq \mu (x,y]$. By the definition of $G_I$,
we immediately get  that for $x<y\leq \xi_I$ or for $\xi_I\leq
x<y$, $|G_I(y)-G_I(x)|= \mu (x,y]$. In the case $x< \xi_I<y$, we
may assume that $G_I(x)<G_I(y)$ (the other case is similarly
treated). Then there is $x<z<\xi_I$ with $G_I(z)=G_I(y)$ and so
$|G_I(y)-G_I(x)|=|G_I(z)-G_I(x)|=\mu (x,z]<\mu (x,z].$\\
(ii) As above it suffices to show that
$|G_{I_2}(x_2)-G_{I_1}(x_1)|\leq \mu(x_1,x_2]$. Let $b_{1}$ be the
right  end of $I_1$ and   $a_{2}$ be the left end of $I_2$. Then
$G_{I_1}(b_{1})=G_{I_1}(a_{2})=0$ and by part (i)
$|G_{I_2}(x_2)-G_{I_1}(x_1)|=|G_{I_2}(x_2)-G_{I_1}(b_{I_1})+G_{I_1}(a_{2})
-G_{I_1}(x_1)|\leq |G_{I_2}(x_2)-G_{I_1}(b_{1})|+|G_{I_1}(a_{2})
-G_{I_1}(x_1)|\leq \mu(x,b_{1}]+\mu(a_{2},y]\leq
\mu(x,y]$.\end{proof}
 \noindent\textit{Claim} 2.  Let $H_n=\sum_{i=1}^{2^{n+1}} H_{I_i^n}$, where
 $I_i^n=[\frac{i-1}{2^n}, \frac{i+1}{2^n}]$,  $n\geq 0$,  $1\leq i\leq 2^n$.
\begin{enumerate}
\item [(i)]  $H_n\in V_2^0$ and
$\|H_n\|_{\infty}=\sqrt{2^{-(n+1)}\mu[0,1]}$.
 \item [(ii)]  For all
$\mathcal{I}\in\mathcal{A}$, $\textit{v}_2^2(H_n,\mathcal{I})\leq
\mu(\bigcup\mathcal{I})$. Therefore $\|H_n\|_{V_2^0}\leq
\sqrt{\mu[0,1]}.$ \item [(iii)] Let
$\mathcal{P}_n=\{i/2^n\}_{i=0}^{2^n}\cup\{\xi_{I_i^n}\}_{i=0}^{2^n}$.
 Then   $\textit{v}_2^2(H_n,\mathcal{P}_n\cap[0,t])=\mu[0,t]$, for all $t\in \mathcal{P}_n$.
 \end{enumerate}
\begin{proof} Straightforward by Claim 1.
\end{proof}
 \noindent\textit{Claim} 3.  There exist a subsequence $(H_{n_i})_i$ of $(H_n)_n$
  and a  strictly decreasing null sequence  $(\delta_i)_{i=0}^\infty$, with $\delta_0=0$,
satisfying the following.
\begin{enumerate}
\item [(i)] $\sum_i\|H_{n_i}\|_\infty<\infty$.
\item [(iii)]  $(H_{n_i})_i$ is
$((2^{-i})_{i=1}^\infty,(\delta_i)_{i=0}^\infty)$-biorthogonal.
\item[(iii)]  For every  $i\in\nn$,
$\delta_i<\|\mathcal{P}_{n_i}\|_{\min}\leq
\|\mathcal{P}_{n_i}\|_{\max}\leq\delta_{i-1}$. \end{enumerate}
\begin{proof} We inductively define a strictly increasing sequence
$n_1<n_2<...$ in $\nn$  and a strictly
 decreasing sequence $0=\delta_0>\delta_1>\delta_2>...$, such that for every $ k\geq 1$, the following hold.
\begin{enumerate}
\item [(1)] For all $1\leq i\leq k$, $\|H_{n_i}\|_\infty<2^{-i}$.
\item[(2)] The finite sequence $(H_{n_i})_{i=1}^k$ is
$((2^{-i})_{i=1}^k,(\delta_i)_{i=0}^{k-1})$-biorthogonal.
\item[(3)] For all $1\leq i\leq k$,
$\delta_i<\|\mathcal{P}_{n_i}\|_{\min}\leq
\|\mathcal{P}_{n_i}\|_{\max}\leq\delta_{i-1}$. \item[(4)]  For
every $\mathcal{I}\in\mathcal{A}$ with
$\|\mathcal{I}\|_{\max}\leq\delta_k$,
$\sum_{i=1}^{k}\textit{v}_2(H_i,\mathcal{I})<2^{-(k+1)}$.
\end{enumerate}

The general inductive step of the construction goes as follows.
 Suppose that for some $k\geq 1$, we have chosen $(n_i)_{i\leq k}$
and $(\delta_i)_{i\leq k}$ satisfying  the above.
Applying   Lemma \ref{bio}
for  $\ee_{k+1}= 2^{-(k+1)}$, we have that there exists  $\ee>0$ such that for every $H\in V_2^0$ with
  $\|H\|_{V_2^0}<\ee$, the sequence $(H_{n_1},...,H_{n_k}, H)$ is
  $((2^{-i})_{i=1}^{k+1},(\delta_i)_{i=0}^{k})$-biorthogonal.
 By Claim 2, we have that  $\lim \|H_n\|_\infty=0$ and $\lim\|\mathcal{P}_n\|_{\max} =0$.
  Hence
 we may choose $n_{k+1}>n_k$ such that $\|H_{n_{k+1}}\|_\infty<\min \{2^{-{k+1}}, \ee\}$
 and $\|\mathcal{P}_{n_{k+1}}\|_{\max}\leq\delta_k$.
Finally  we choose
$0<\delta_{k+1}<\|\mathcal{P}_{n_{k+1}}\|_{\min}$ such that for
every $\mathcal{I}\in\mathcal{A}$ with
$\|\mathcal{I}\|_{\max}\leq\delta_{k+1}$,
$\sum_{i=1}^{k+1}\textit{v}_2(H_i,\mathcal{I})<2^{-(k+2)}$ and the
proof of the inductive step of the construction is complete.
\end{proof}
\noindent\textit{Claim} 4. Let  $f=\sum_i H_{n_i}$.
 Then $f\in V_2\cap C[0,1]$ and $\mu_f=\mu$.
 \begin{proof} Since $\sum_i\|H_{n_i}\|_\infty<\infty$ and $H_{n_i}$ are
continuous we have that $f\in C[0,1]$.  By (ii) of Claim 2,
$(H_n)_n$ is a bounded (by $M=\sqrt{\mu[0,1]}$) sequence in $V_2^0$.
Moreover by Claim 3, $(H_{n_i})_i$ is biorthogonal and so by  Lemma
\ref{lul2} and (ii) of Claim 2, we have that for every
$\mathcal{I}\in \mathcal{A}$
\[\begin{split} \textit{v}_2^2(f,\mathcal{I})\leq \sum_i\mu(\cup \mathcal{I}^{(i)})+
2\sqrt{\mu[0,1]}+1\leq (\sqrt{\mu[0,1]}+1)^2\end{split}\]
 and therefore $f\in V_2$. To prove that  $\mu_f=\mu$, let $D=\bigcup_{i}D_{n_i}$ where for all $i\in\nn$,
$D_{n_i}=\{m2^{-n_i}:\;0\leq m\leq 2^{n_i}\}$. Since $D$ is dense in $[0,1]$, it suffices to
show that $\mu_f[0,t]=\mu[0,t],$ for all
$t\in D$.

Fix $i_0$ and $0\leq m_0\leq 2^{n_{i_0}}$ and let
$t=m_0/2^{n_{i_0}}$. By the definition of $(\mathcal{P}_n)_n$, we
have that  for all $j\geq i_0$, $t\in \mathcal{P}_{n_i}$,
 and so   by (iii) of Claim 2,
 \[\mu[0,t]=\textit{v}_2^2(H_{n_j},\mathcal{P}_{n_j}\cap[0,t])\]
Hence by (iii) of Claim 3, for all  $j\geq i_0$,
\[\begin{split}\textit{v}_2(f,\mathcal{P}_{n_j}\cap [0,t])&
=\textit{v}_2(H_{n_j}+\sum_{i\neq j} H_{n_i},\mathcal{P}_{n_j}\cap
[0,t])\geq \mu[0,t] -2^{-j}\end{split}\] hence since $\lim_j
\|\mathcal{P}_{n_j}\|_{\max}=0$, $\mu_f[0,t]\geq \lim_j
\textit{v}_2^2(f,\mathcal{P}_{n_j}\cap[0,t])\geq\mu[0,t]$.

It remains to show that $\mu_f[0,t]\leq \mu[0,t]$. Since $\lim
\delta_k=0$, we have that
\begin{equation}\label{m2}\mu_f[0,t]=\limsup_k\{\textit{v}_2^2(f,\mathcal{P}):\;\mathcal{P}\subseteq
[0,t], \|\mathcal{P}\|_{\max}<\delta_k\}\end{equation} Fix $k\geq 1$
and $\mathcal{P}\subseteq [0,t]$  such that
$\|\mathcal{P}\|_{\max}\leq \delta_{k-1}$. Let
$\mathcal{I}=\mathcal{I}_\mathcal{P}$  be the corresponding family
of intervals with endpoints successive points of $\mathcal{P}$. Then
$\mathcal{I}=\bigcup_{j\geq k}\mathcal{I}^{(j)}$, where
$\mathcal{I}^{(j)}=\{I\in\mathcal{I}:\delta_j<|I|\leq
\delta_{j-1}\}$, and
\begin{equation}\label{m1}\begin{split}\textit{v}_2^2(f,\mathcal{P})=\textit{v}_2^2(f,\mathcal{I})=
\textit{v}_2^2(f,\bigcup_{j\geq k}\mathcal{I}^{(j)})=\sum_{j\geq
k}\textit{v}_2^2(f,\mathcal{I}^{(j)})\end{split}\end{equation}
Moreover by (ii) of Claim 2,
\[\begin{split}\textit{v}_2(f,\mathcal{I}^{(j)})\leq\textit{v}_2(H_{n_j}+\sum_{i\neq
j}H_{n_i}, \mathcal{I}^{(j)})\leq \textit{v}_2(H_{n_j},
\mathcal{I}^{(j)})+2^{-j}\leq \mu(\cup \mathcal{I}^{(j)})+
2^{-j}\end{split}\] Hence using (\ref{m1}), we obtain that
\[\textit{v}_2^2(f,\mathcal{P})\leq\sum_{j\geq k}\mu(\cup \mathcal{I}^{(j)})+
\sum_{j\geq k}2^{-j}=\mu(\cup \mathcal{I})+2^{k-1}=\mu
[0,t]+2^{k-1}\]and therefore by (\ref{m2}),  $\mu_f[0,t]\leq
\mu[0,t]$.
\end{proof}
By Claim 4 the proof of the theorem is complete.\end{proof}

\begin{rem}\label{Rad} Let us  present here  a concrete example
which illustrates the
 method of the proof of Theorem \ref{tm}. Let $\lambda$ be the Lebesgue measure
 on $[0,1]$ and let $(R_n)_n$ defined by
 \begin{equation}\label{Radem} R_n(t)=2^{n/2}\int_0^tr_n(x) dx,\end{equation} where $(r_n)_n$ is the sequence of
 Rademacher functions. As in Claims 1 and 2 of Theorem \ref{tm}, it can be shown that
\begin{enumerate}\item[(i)] $\textit{v}_2^2(R_n,\mathcal{I})\leq
\lambda(\cup \mathcal{I})$. \item[(ii)] For every $m\geq n$ and
$1\leq i\leq 2^n$, $\textit{v}_2^2(R_n,
\mathcal{P}_m\cap[0,\frac{i}{2^n}])=\lambda[0,\frac{i}{2^n}]$.
\end{enumerate}
where here  $\mathcal{P}_n=\{\frac{i}{2^n}:0\leq i\leq 2^n\}$.
Then as in Claim 3 we may show that there is a subsequence
$(R_{n_i})_i$  of $(R_n)_n$ such that the  sum $f=\sum_iR_{n_i}$
satisfies that $\mu_f=\lambda$. We note that as it has been stated
in \cite{LS}, the above defined sequence $(R_n)_n$ contains
subsequences equivalent to $c_0$ basis. In the sequel (Corollary
\ref{corrad}) we shall provide a proof of this statement.\end{rem}

\subsection{On the points of non differentiability  of  functions in $V_2$.}
\begin{lem}\label{ldd}
Let $f\in V_2$ and let  $(\mathcal{P}_n)_n$ be a sequence of
finite  subsets of  $[0,1]$ such that
$\lim\|\mathcal{P}_n\|_{\max}=0$ and
$\lim\textit{v}_2^2(f,\mathcal{P}_n)=\mu_f[0,1]$. Then for every
sequence $(\mathcal{I}_n)_n$  in $\mathcal{A}$ such that
$\mathcal{I}_n\subseteq \mathcal{I}_{\mathcal{P}_n}$ and
$(\textit{v}_2^2(f,\mathcal{I}_n))_n$ converges, we have that
$(\mu_f(\cup\mathcal{I}_n))_n$ also converges and $\lim
\textit{v}_2^2(f,\mathcal{I}_n)=\lim \mu_f(\cup\mathcal{I}_n)$.
\end{lem}
\begin{proof}
 Let $\alpha=\lim
\textit{v}_2^2(f,\mathcal{I}_n)$ and assume that
$(\mu_f(\cup\mathcal{I}_n))_n$ does not converge to $\alpha$. Then
by passing to a subsequence, we may suppose that $\lim
\mu_f(\cup\mathcal{I}_n)= \beta\neq \alpha.$ Let
$\mathcal{J}_n=\mathcal{I}_{\mathcal{P}_n}\setminus\mathcal{I}_n$.
Then \[\lim
\textit{v}_2^2(f,\mathcal{J}_n)=\mu_f[0,1]-\alpha\;\text{ and}\;
\lim \mu_f(\cup\mathcal{J}_n)=\mu_f[0,1]-\beta\] Since
$\mathcal{I}_n$ and $\mathcal{J}_n$ consist of open intervals of
$[0,1]$, we can choose $\mathcal{I}_n' \preceq\mathcal{I}_n$ and
$\mathcal{J}_n '\preceq\mathcal{J}_n$ such that
$|\mu_f(\cup\mathcal{I}_n)-\textit{v}_2^2(f,\mathcal{I}_n')|<1/n$
and
$|\mu_f(\cup\mathcal{J}_n)-\textit{v}_2^2(f,\mathcal{J}_n')|<1/n$.
Therefore we get that
\[\lim\textit{v}_2^2(f,\mathcal{I}_n')= \beta\;\;
\text{and}\;\;\lim \textit{v}_2^2(f,\mathcal{J}_n')=
\mu_f[0,1]-\beta\] Since  $\mathcal{I}_n, \mathcal{J}_n'$  are
disjoint and $\lim \|\mathcal{I}_n\|_{\max}=\lim
\|\mathcal{J}_n'\|_{\max}=0$ we obtain that
\[\mu_f[0,1]\geq \lim \textit{v}_2^2(f,\mathcal{I}_n\cup\mathcal{J}_n')
 =\alpha+(\mu_f[0,1]-\beta),\] which implies that $\beta\geq \alpha$. Similarly,
\[\mu_f[0,1]\geq\lim \textit{v}_2^2(f,\mathcal{I}_n'\cup\mathcal{J}_n)
=\beta+(\mu_f[0,1]-\alpha),\] which gives that $\alpha\geq \beta$.
Hence $\alpha=\beta$ which is a contradiction.
\end{proof}
\begin{thm}\label{nondif} Let $f\in V_2$. Then the set
of all points $x\in [0,1]$ such that $f$ is differentiable at $x$
has $\mu_f$-measure zero.
\end{thm}
\begin{proof} Let  $\mathcal{P}_n=\{0=t^n_0<...<t_{k_n}^n=1\}\subseteq
[0,1]$  such that $\lim \|\mathcal{P}_n\|_{\max}=0$ and $\lim
\textit{v}_2^2(f,\mathcal{P}_n)=\mu_f([0,1])$. It suffices to show
that for every $C>0$, $\mu_f(A_C)=0$, where \[A_C=\{x\in
[0,1]:\;\exists f'(x)\;\text{and} \;|f'(x)|< C\}\] Fix  $C>0$ and
for every $k\in\nn$ let $A_C^k$ to be the set of all
$x\in(0,1)\setminus \bigcup_n\mathcal{P}_n$ such that for every
$y,z\in (x-1/k,x+1/k)$ with $0\leq y<x<z\leq 1$,
$|\frac{f(z)-f(y)}{z-y}|< C$.

Notice that \[A_C\subseteq \Big(\bigcup_n \mathcal{P}_n\setminus
D_f\Big) \cup \bigcup_{k=1}^\infty A_C^k\] By Proposition
\ref{discr.}, we get that  $\mu_f(\cup_n \mathcal{P}_n\setminus
D_f)=0$ and therefore it remains to show that for every $k\in\nn$,
$\mu_f(A_C^k)=0$. To this end, fix  $k\in\nn$.  Then  for every
$x\in A^k_C$ and  $n\in\nn$ there exists $0\leq i\leq k_n-1$ such
that $x\in(t_i^n,t_{i+1}^n)$. Since
$\lim\|\mathcal{P}_n\|_{\max}=0$ there is $n_0$ such that for all
$n\geq n_0$, $t_{i+1}^n-t_i^n<1/k$. For $n\geq n_0$ let $F_n$ to
be the set of all $0\leq i\leq k_n-1$ such that $A_C^k\cap
(t_i^n,t_{i+1}^n)\neq \emptyset$ and let
$\mathcal{I}_n=((t_i^n,t_{i+1}^n))_{i\in F_n}$. Then
$A_C^k\subseteq \bigcup\mathcal{I}_n$ and therefore
\[\textit{v}_2^2(f,\mathcal{I}_n)=\sum_{i\in
F_n}|f(t_{i+1}^n)-f(t_i^n)|^2\leq C^2\sum_{i\in F_n}|t_{i+1}^n-
t_i^n|^2\leq C^2 \max_{i\in F_n}|t_{i+1}^n- t_i^n|\] Hence $\lim
\textit{v}_2^2(f,\mathcal{I}_n)=0$. By Lemma \ref{ldd}, $\lim\mu_f
(\bigcup \mathcal{I}_n)=0$ and so $\mu_f(A_C^k)=0$.
\end{proof}

\begin{cor} Let $f\in V_2 \cap C[0,1]$. If the set
of all points $x\in [0,1]$ such that $f$ is not  differentiable at
$x$ is countable then $f\in V_2^0$.  Moreover if $f\in
(V_2\setminus V_2^0)\cap C[0,1]$ then the set of all non
differentiability points of $f$ contains a perfect set.

\end{cor}
\begin{proof} Let $B$ be the set of all $x\in [0,1]$
such that $f$ is not differentiable at $x$.  By Theorem
\ref{nondif}, we have that $\mu_f(B)=\mu_f[0,1]$.  Also since
$f\in C[0,1]$, $\mu_f$ is continuous. Therefore if $B$ is
countable then $\mu_f[0,1]=0$ and so  $f\in V_2^0$. In the case
$f\in (V_2\setminus V_2^0)\cap C[0,1]$, $\mu_f(B)>0$ and so $B$
contains a perfect set.
\end{proof}

\section{Geometric properties of the measure.} In this section we
mainly concern to connect the norm of the measure $\mu_f$ with the
distance of $f$ from $V_2^0$. This requires first some results
from \cite{AK}, included in the first subsection,  related to the
oscillation function $\widetilde{osc} f$
 defined by A. Kechris and A. Louveau \cite{KL} and
 further studied by H. P. Rosenthal in \cite{R}. The second
 subsection contains the statement and the proof of the basic
 inequality and
in the third subsection  we use  these geometric properties of the
measure to obtain optimal  approximations for the functions of
$V_2\setminus V_2^0$.

\subsection{The oscillation function.} Recall that  for a
function  $f:K\rightarrow\rr$ where $K$ is a compact metric space,
$\widetilde{osc}_K f$,  is defined as follows.   For every $t\in
V\subseteq K$  let $s(V,t)=\sup\{|f(x)-f(t)|:x\in V\}$. Then for
each $t\in K$, \[\widetilde{osc}_K f(t)=\inf\{s(V,t):V \text{open
neighborhood of t}\}\]

It can be easily shown that for every sequence  $(f_n)_{n}$
  of continuous real valued functions on $K$  pointwise converging to a function
  $f$ and every $\varepsilon>0$ there exists $n_0\in \nn$ such that
  for every $n\geq n_0$,
  $\|\widetilde{osc}_Kf\|_\infty-\varepsilon<\|f_n-f\|_\infty$.

The next lemma is included in the more general Lemma 1.2 in
\cite{AK} and shows that passing to convex blocks the above
inequality can be reversed.
\begin{lem} \label{approx} Let $(f_n)_n$ be a uniformly
bounded sequence of continuous real valued functions
 on a compact metric space $K$ pointwise converging to a function
 $f$. Then for every null sequence of positive reals $(\delta_n)_n$ there exist a convex block sequence $(g_n)_n$ of
 $(f_n)_n$ such that
 $\|g_n-f\|_\infty< \|\widetilde{osc}_K f\|_\infty+\delta_n$.
\end{lem}

\begin{cor} \label{cap} Let $E$ be a separable Banach space, $X$ be a subspace of $E$
 and  $K$ be a weak$^{*}$-compact subset  of  $B_{E^*}$ which is
 $1$-norming for $E^{**}$ (that is for all $x^{**}\in E^{**}$,
$\|x^{**}\|=\sup_{x^*\in K}|x^{**}(x^*)|$). Then for all
$x^{**}\in X^{**}$ which are  weak$^*$-limits of sequences in $X$,
we have that
\[dist(x^{**}, X)\leq\|\widetilde{osc}_{K} x^{**}\|_\infty\] In particular this holds
for all $x^{**}\in X^{**}$ if $\ell_1$ is not embedded into $X$.
\end{cor}
\begin{proof} Let $K$ be a weak$^*$ subset of  $B_{E^*}$ which is
 $1$-norming for $X^{**}$. Let $x^{**}\in X^{**}$ be a the  weak$^*$- limit of a sequence $(x_n)_n$ in $X$.
Denoting again by $x_n$ and $x^{**}$ the restrictions of $x_n$ and
$x^{**}$ on $K$, we have that
 $(x_n)_n$ is
a uniformly bounded  sequence of continuous functions on the
compact metric space $K$,
 pointwise convergent to $x^{**}$. Therefore, by
  Lemma \ref{approx}
 there is a convex block sequence $(y_n)_n$ of $(x_n)_n$ such that for all $n\in\nn$,
 \[\sup_{x^{*}\in K}|y_n(x^*)-x^{**}(x^{*})|\leq \|\widetilde{osc}_K x^{**}\|_{\infty}+1/n\]
 Since $K$ is $1$-norming for $E^{**}$, we have that the left side of the above inequality is the norm of $y_n-x^{**}$.
 Hence
\[dist(x^{**}, X)\leq \|y_n-x^{**}\|\leq\|\widetilde{osc}_{K} x^{**}\|_\infty+1/n,\]
for all $n\in \nn$ and the result follows. Finally if $\ell_1$
does not embed into $X$ then by Odell-Rosenthal's theorem
\cite{OR}, all $x^{**}\in X^{**}$ are $w^*$-limits of sequences of
$X$.
\end{proof}

\begin{rem}\label{Rem3} Notice  that for every $x^{**}\in X^{**}$ and every $x\in
X$, $\|\widetilde{osc}_\mathcal{K}
x^{**}\|_\infty=\|\widetilde{osc}_\mathcal{K}
(x^{**}-x)\|_\infty\leq 2\|x^{**}-x\|$ and so
$\|\widetilde{osc}_\mathcal{K} x^{**}\|_\infty\leq 2
dist(x^{**},X)$. Hence by Corollary \ref{cap}, we have that for
every subspace $X$ of $E$ not containing $\ell_1$,
$\|\widetilde{osc}_\mathcal{K} x^{**}\|_\infty$ is an equivalent
norm on the quotient space $X^{**}/X$.\end{rem}

\subsection{Connection of the measure $\mu_f$ with the distance
of $f$ from $V_2^0$.} Returning to $V_2$, we recall that  the set
$\mathcal{K}\subseteq (V_2^0)^*$  of all $x^*$
 of the form $x^*=\sum_i\alpha_i(\delta_{s_i}-\delta_{t_i})$ where $\big((s_i,t_i)\big)_i$
  is a sequence of pairwise disjoint intervals of $[0,1]$ and $\sum_i\alpha_i^2\leq 1$,
  is a $w^*$-compact subset of $ B_{(V_2^0)^*}$ which is  1-norming for $V_2$
  (\cite{AMP}).

\begin{thm} \label{osc-meas} Let $X$ be a subspace of $V_2^0$. Then for
 every  $f\in X^{**}$, \[\sqrt{\|\mu_f\|}\leq dist (f, V_2^0)\leq dist(f,X)\leq
\|\widetilde{osc}_\mathcal{K}
f\|_{\infty}\leq\sqrt{\|\mu_f\|}+2\sqrt{\|\mu_f^d\|}\]\end{thm}

\begin{proof} By Proposition \ref{basineq}, we have that for every $g\in V_2^0$, $\mu_{f+g}=\mu_f$. Hence
for every $g\in V_2^0$,
$\|\mu_f\|=\|\mu_{f+g}\|\leq\|f+g\|^2_{V_2}$ which gives that
$\sqrt{\|\mu_f\|}\leq dist(f,V_2^0)\leq dist(f, X)$. Since $K$ is
a weak$^*$-compact subset of $(V_2^0)^*$ which is $1$-norming for
$V_2$ and $\ell_1$ is not embedded into $V_2^0$, by Corollary
\ref{cap}, we have that $dist(f, X)\leq
\|\widetilde{osc}_\mathcal{K} f\|_{\infty}$ and so it remains to
prove that $\|\widetilde{osc}_\mathcal{K}
f\|_{\infty}\leq\sqrt{\|\mu_f\|}+2\sqrt{\|\mu_f^d\|}$. To show
this, let $x^{*}\in\mathcal{K}$ and let $\{x_n^{*}\}_{n}$ be a
sequence in $\mathcal{K}$, $w^{*}$-converging to $x^{*}$ such that
$\widetilde{osc}_{\mathcal{K}}f(x^{*})=\lim_{n}|f(x_n^{*})-f(x^{*})|$
(clearly there exists such a sequence). Notice that for every
$y^*=\sum_i\alpha_i(\delta_{s_i}-\delta_{t_i})\in\mathcal{K}$ and
every $f\in V_2$, $f(y^*)=\sum_i\alpha_i(f(s_i)-f(t_i))$ where the
series $ \sum_i\alpha_i(f(s_i)-f(t_i))$ is absolutely convergent.
So we may  reorder  each $x_n^*$ and in this way  we may assume
that
$x_n^{*}=\sum_{i=1}^{\infty}\alpha_i^n(\delta_{s_i^n}-\delta_{t_i^n})$
where   for every $i$, $t_i^n-s_i^n\geq t_{i+1}^n-s_{i+1}^n$.
Moreover by passing to a subsequence we may also suppose  that for
each $i\in\nn,$ the sequences $\{s_i^n\}_{n}$, $\{t_i^n\}_{n}$ are
monotone and that
 $\alpha_i^n\rightarrow\alpha_i,\, s_i^n\rightarrow s_i,$ and
$ t_i^n\rightarrow t_i$.

Therefore
$x^{*}=\sum_{i=1}^{\infty}\alpha_i(\delta_{s_i}-\delta_{t_i})$
with $t_i-s_i\geq t_{i+1}-s_{i+1}$ and so  since $((s_i,t_i))_i$
consists of pairwise disjoint open intervals in $[0,1]$ of
decreasing length, $t_i-s_i\leq 1/i$. Also  by the monotonicity of
$(s_i^n)_n$, $(t_i^n)_n$,
 there are $\ee_{t_i},\ee_{s_i}\in\{0,+,-\}$,
 such that $\lim_n f(t_i^n)=f(t_i^{\ee_{t_i}})$ and $\lim_n f(s_i^n)=f(s_i^{\ee_{s_i}})$.

 Let $\ee>0$. We choose $i_1\in\nn$ such that $\mu_f[0,1]\leq \widetilde{\mu}_{f,1/i_1}[0,1]
<\mu_f[0,1]+\ee^2$ and $\sum_{i=i_1+1}^{\infty}|\alpha_i|^2<
\ee^2/\|f\|_{V_2}^2$. We set
\[x_1^{*}=\sum_{i=1}^{i_1}\alpha_i(\delta_{s_i}-\delta_{t_i})\;\;\;\;\;\;\text{and}\;\;\;\;
x_2^{*}=\sum_{i=i_1+1}^{\infty}\alpha_i(\delta_{s_i}-\delta_{t_i})\]
and  for each $n\in\nn$, let
\[x_{n,1}^{*}=\sum_{i=1}^{i_1}\alpha_i^n(\delta_{s_i^n}-\delta_{t_i^n})\;\;\;\;\;\;\text{and}\;\;\;\;\;\;\;
x_{n,2}^{*}=\sum_{i=i_1+1}^{\infty}\alpha_i^n(\delta_{s_i^n}-\delta_{t_i^n})\]
Then by  Cauchy-Schwartz inequality we get that
$|f(x_{2}^{*})|^2\leq\ee^2$ and
 $|f(x_{n,2}^{*})|^2\leq\mu_f[0,1]+\ee^2$, for all $n\in\nn$.
Therefore
\[\begin{split}&\widetilde{osc}_{\mathcal{K}}f(x^{*})
=\lim_{n}|f(x_n^{*})-f(x^{*})|\leq
|\lim_{n}f(x_{n,1}^{*})-f(x_1^{*})|+\overline{\lim_n}
|f(x_{n,2}^{*})-f(x_2^{*})|\\&
\leq\Big|\sum_{i=1}^{i_1}\alpha_i(f(s_i^{\ee_{s_i}})-f(s_i))\Big|+
\Big|\sum_{i=1}^{i_1}\alpha_i(f(t_i^{\ee_{t_i}})-f(t_i))\Big|
+\overline{\lim_n} |f(x_{n,2}^{*})|+ |f(x_{2}^{*})|\\& \leq
\Big(\sum_{i=1}^{i_1}|f(s_i^{\ee_{s_i}})-f(s_i)|^2\Big)^{1/2}+
\Big(\sum_{i=1}^{i_1}|f(t_i^{\ee_{t_i}})-f(t_i)|^2\Big)^{1/2}
+(\sqrt{\mu_f[0,1]}+\ee)+\ee\\&\leq
\Big(\sum_{i=1}^{i_1}\mu_f(\{s_i\})\Big)^{1/2}+\Big(\sum_{i=1}^{i_1}\mu_f(\{t_i\})\Big)^{1/2}
+\sqrt{\mu_f[0,1]}+2\ee.
\end{split}\]Hence  for every $\ee>0$, $\|\widetilde{osc}_K
f\|_{\infty}\leq\sqrt{\|\mu_f\|}+2\sqrt{\|\mu_f^d\|}+2\ee$ and
the conclusion follows.\end{proof} Since $\mu_f^d=0$  if  $f$ is
continuous, we easily get the following.
\begin{cor}\label{corcon} For every subspace $X$ of $V_2^0$ and every $f\in X^{**}\cap C[0,1]$,
  \[dist(f,V_2^0)=dist(f,X)=\|\widetilde{osc}_\mathcal{K}f\|_\infty=\sqrt{\|\mu_f\|}\]
\end{cor}

\begin{rem} Let us note that there exist    non-continuous functions $f\in
V_2$ satisfying  the proper inequalities $\sqrt{\|\mu_f\|}< dist
(f, V_2^0)< \|\widetilde{osc}_{\mathcal{K}}f\|_\infty$. For
example, it can be easily shown that  for $0<s<t<1$ and
$f=\chi_{[s,t]}+2\chi_{(t,1]}$, we have that
 $\|\widetilde{osc}_{\mathcal{K}}f\|_\infty= 2$, $\mu_f$ is the
sum of the Dirac measures on $s$ and $t$ (so
$\sqrt{\mu_f}=\sqrt{\mu_f^d}=\sqrt{2}$) and
$\sqrt{2}<dist(f,V_2^0)< 2$.\end{rem}
\subsection{Optimal  approximation of functions of $V_2\setminus V_2^0$.}
\begin{lem}\label{luse} Let $f\in V_2\setminus V_2^0$ and let $(f_n)_n$ be
a  bounded sequence in $V_2^0$ pointwise convergent to $f$. Then
for every sequence $(\ee_n)_n$ of positive real numbers
there exists a  convex block sequence  $(h_n)_n$ of $(f_n)_n$
satisfying the following properties.
\begin{enumerate}
\item [(i)]  $\|h_n-f\|_\infty\leq \|\widetilde{osc}_{[0,1]}
f\|_{\infty}+\ee_n$. \item [(ii)] For every
$\mathcal{I}\in\mathcal{F}([0,1]\setminus D_f)$,
$\textit{v}_2^2(h_n-f,\mathcal{I})\leq\mu_f(\cup\mathcal{I})+8\|f\|_{V_2}
\sqrt{\|\mu_f^d\|}+\ee_n$.
\end{enumerate}
\end{lem}
\begin{proof}  We may assume that $\ee_n<1$. Let also  $\delta_n= (1+6\|f\|_{V_2})^{-1}\ee_n$.
 By our assumptions we have that
$(f_n)_n$ is a uniformly bounded sequence of continuous functions
on $[0,1]$ pointwise convergent to $f$. Hence by Lemma
\ref{approx} (for $K=[0,1]$), there exists a convex block sequence
$(g_n)_n $ of $(f_n)_n$ such that
 \begin{equation}\label{a1}\|g_n-f\|_\infty\leq
 \|\widetilde{osc}_{[0,1]}
f\|_{\infty}+\delta_n,\end{equation}  Now $(g_n)_n$ is a uniformly
bounded sequence of w$^*$- continuous functions on the compact
metric space $\mathcal{K}$ and so by the same lemma and Theorem
\ref{osc-meas}, there exists a convex block sequence $(h_n)_n$ of
$(g_n)_n$ such that
\begin{equation}\label{a2}\|h_n-f\|_{V_2}\leq \|\widetilde{osc}_{\mathcal{K}}
f\|_{\infty}+\dd_n\leq
\sqrt{\|\mu_f\|}+2\sqrt{\|\mu_f^d\|}+\delta_n\end{equation} It is
clear that by (\ref{a1}) we have that
\begin{equation}\label{a3} \|h_n-f\|_\infty\leq
 \|\widetilde{osc}_{[0,1]}
f\|_{\infty}+\delta_n
\end{equation}
Moreover, using that $\|\mu_f^d\|=\sqrt
{\|\mu_f^d\|}\sqrt{\|\mu_f^d\|}\leq \|f\|_{V_2}\sqrt{\|\mu_f^d\|}$
and taking squares in  (\ref{a2}) we easily  obtain that
\begin{equation}
\label{ldom3} \|h_n-f\|_{V_2}^2\leq
\|\mu_f\|+8\|f\|_{V_2}\sqrt{\|\mu_f^d\|}+\ee_n
\end{equation} Let $\mathcal{I}\in\mathcal{F}([0,1]\setminus D_f)$. Then  $[0,1]\setminus \cup\mathcal{I}=\cup_{i=1}^m I_i'$ where each
$I_i'$ is a non trivial open interval in $[0,1]$. By Proposition
\ref{pm}, $\mu_f(I_i')=\widetilde{\mu}_f(I_i')$ and so for each
$1\leq i\leq m$, we can choose a sequence $(\mathcal{P}^i_k)_k$ of finite subsets
of $I_i'$, such that
 $\lim_k\|\mathcal{P}_k^i\|_{\max}=0$ and
$\lim_k\textit{v}_2^2(f,\mathcal{P}_k^i)=\mu_f(I_i')$. Since
$h_n\in V_2^0$,  for all $1\leq i\leq m$,
\begin{equation}\label{a4}\lim_k\textit{v}_2^2(h_n-f,\mathcal{P}_k^i)=\lim_k
\textit{v}_2^2(f,\mathcal{P}_k^i)=\mu_f(I_i')\end{equation} Also
setting $\mathcal{I}'=(I_i')_{i=1}^m$,
$\|\mu_f\|=\mu_f[0,1]=\mu_f(\cup\mathcal{I})+\mu_f(\cup
\mathcal{I}')$. Hence    (\ref{ldom3}) gives that
\[\textit{v}_2^2(h_n-f,\mathcal{I})+\sum_{i=1}^m\textit{v}_2^2(h_n-f,\mathcal{P}_k^i)
\leq\mu_f(\cup\mathcal{I})+\mu_f(\cup\mathcal{I}')+
8\|f\|_{V_2}\sqrt{\|\mu_f^d\|}+\ee_n\] Letting
$k\rightarrow\infty$ and using  (\ref{a4}), part (ii) of the lemma follows.
\end{proof}

\begin{prop}\label{ldom}
Let $X$ be a subspace of $V_2^0$, $f\in X^{**}\setminus X$ and  $(f_n)_n$ be a bounded sequence
in $X$ pointwise
convergent to $f$.  Then for every $0<\delta<dist(f,X)$ and
 for every sequence $(\ee_n)_n$ of positive real numbers there exist
a  convex block sequence $(h_n)_n$ of  $(f_n)_n$ such that for all $n<m$ the following properties are satisfied.
\begin{enumerate}
\item [(i)] $\delta< \|h_m-h_n\|_{V_2}\leq 2M$, where $M=\sup_n
\|f_n\|_{V_2} $. \item[(ii)] $\|h_m-h_n\|_{\infty}\leq
2\|\widetilde{osc}_{[0,1]} f\|_{\infty}+\ee_n\leq
4\|f\|_\infty+\ee_n.$  \item [(iii)] For every
$\mathcal{I}\in\mathcal{A}$,
$\textit{v}_2^2(h_m-h_n,\mathcal{I})\leq 4\mu_f(\cup\mathcal{I})+
32\|f\|_{V_2}\sqrt{\|\mu_f^d\|}+\ee_n$.
\end{enumerate}
Moreover given $k\in\nn$ and an open subset $V$
 of $[0,1]$ with    $\mu_f(V)>\theta>0$  there exist
$\mathcal{J}\in\mathcal{A}$ with $\cup\mathcal{J}\subseteq V$ and $l>k$ such that
\begin{enumerate}
\item[(iv)] $\textit{v}_2^2(h_l-h_k,\mathcal{J})>\theta$.
\end{enumerate}  \end{prop}
\begin{proof}
Let $(\ee_n')$ be a decreasing sequence of positive real numbers
with $\ee_n'<4\ee_n$. Let $(h_n)_n$ be the convex block sequence
of $(f_n)_n$ resulting from  Lemma \ref{luse}. Since $(h_n)_n$ is
w$^*$- convergent to $f$, for every $n\in\nn$ there are finitely
many $m>n$ such that $\|h_m-h_n\|_{V_2}\leq\delta$ (otherwise,
$\|f-h_n\|_{V_2}\leq\delta<dist(f,X)$ which is impossible).
Therefore by passing to a subsequence we may assume that for all
$n<m$, $\delta< \|h_m- h_n\|_{V_2}$. Also since $(h_n)_n$ is a
convex block sequence of  $(f_n)_n,$  $\|h_n\|_{V_2}\leq M$ and so
$\|h_m- h_n\|_{V_2}\leq 2M$.

To show (ii), notice that by (1) above,
\[\|h_m-h_n\|_\infty\leq \|h_n-f\|_\infty+\|h_m-f\|_\infty\leq 2\|\widetilde{osc}_{[0,1]} f\|_{\infty}+\ee_n\leq 4\|f\|_\infty+\ee_n\]
For property  (iii) observe that  since $h_m-h_n$ is a continuous
function on $[0,1]$, and $[0,1]\setminus D_f$ is dense in $[0,1]$,
it suffices to check it   for
$\mathcal{I}\in\mathcal{F}([0,1]\setminus D_f)$. In this case, by
Lemma \ref{luse} (ii), we have that
\[\begin{split}\textit{v}_2^2(h_m-h_n,\mathcal{I})&\leq 2\textit{v}_2^2(h_{m}-f,\mathcal{I})+
2\textit{v}_2^2(h_n-f,\mathcal{I})\\&\leq
4\mu_f(\cup\mathcal{I})+32\|f\|_{V_2}
\sqrt{\|\mu_f^d\|}+\ee_n\end{split}\] Finally fix $k\in\nn$ and
let $V$ be an open subset of $[0,1]$ with  $\mu_f(V)>\theta>0$.
Choose a family $(I_i)_{i=1}^m$ of disjoint open intervals of
$[0,1]$ such that $I_i\subseteq V$ and $\mu_f(\cup_{i=1}^m
I_i)>\theta$ and let
\[0<\ee<\frac{\mu_f(\cup_{i=1}^m I_i)-\theta}{2(1+M)m}\]
Since $h_k\in V_2^0$, there is   some $\dd>0$   such that
\begin{equation}\label{eq13} \sup\{\textit{v}_2(h_k,\mathcal{Q}): \mathcal{Q}\subseteq
[0,1],\, \|\mathcal{Q}\|_{\max}<\dd\}<\ee.\end{equation} For
every $1\leq i\leq m$, there exists $\mathcal{P}_i\subseteq I_i$ with
$\|\mathcal{P}_i\|_{\max}\leq \dd$ and
\begin{equation}\label{eq14}\mu_f(I_i)-\ee<\textit{v}_2^2(f,\mathcal{P}_i)\end{equation} Moreover since
$(h_n)_n$ converges pointwise to $f$, there is $l>k$ such
that for all $1\leq i\leq m$,
\begin{equation}\label{eq15}|\textit{v}_2^2(f,\mathcal{P}_i)-\textit{v}_2^2(h_l,\mathcal{P}_i)|<\ee\end{equation}
Then for every $1\leq i\leq m$ we have that
\[\begin{split}
\textit{v}_2^2(h_l-h_k,\mathcal{P}_i)&\geq
(\textit{v}_2(h_l,\mathcal{P}_i)-\textit{v}_2(h_k,\mathcal{P}_i))^2\geq
\textit{v}_2^2(h_l,\mathcal{P}_i)-2\textit{v}_2(h_k,\mathcal{P}_i)
\textit{v}_2(h_l,\mathcal{P}_i)\\&\geq
\textit{v}_2^2(f,\mathcal{P}_i)-\ee-2\ee\|h_l\|_{V_2}\geq
\mu_f(I_i)-2(1+M)\ee
\end{split}\]
 Therefore, setting $\mathcal{J}=\cup_{i=1}^k\mathcal{I}_{\mathcal{P}_i}$, we obtain that
\[\textit{v}_2^2(h_l-h_k,\mathcal{J})=\sum_{i=1}^{m}\textit{v}_2^2(h_l-h_k,\mathcal{P}_i)
\geq\sum_{i=1}^{m}\mu_f(I_i)-2(1+M)m\ee>\theta.\]
\end{proof}

\begin{rem}\label{Rem5} Notice that since  $\ell_1$ is not embedded into $V_2^0$,
from \cite{OR} and Goldstine's theorem, there is a sequence
$(f_n)_n$ in $X$ pointwise converging to $f$ with
$\|f_n\|_{V_2}\leq \|f\|_{V_2}$. Hence, in Proposition \ref{ldom}
we can assume that $(f_n)_n$ (and thus also $(h_n)_n$) is a
sequence in $X$ with $\|f_n\|_{V_2}\leq \|f\|_{V_2}=M$.\end{rem}

\section{On the embedding of $c_0$ into  subspaces of $V_2^0$.}
In this section we show that every subspace $X$ of $V_2^0$ with $X^*$ separable, $X^{**}$ non separable
and $\mathcal{M}_{X^{**}}=\{\mu_f:f\in X^{**}\}$ separable, contains an isomorphic copy of $c_0$.
 This is the first step towards the proof of the main theorem integrated  in the next section.
\subsection{Sequences of $V_2^0$  dominated by measures.}

\begin{defn} Let $\mu$ be a positive finite Borel measure on $[0, 1]$ and $C,\ee$ be positive constants.
We will say that  a function $G$  of $ V_2^0$   is
$(C,\ee)-$dominated by $\mu$ if  for every
$\mathcal{I}\in\mathcal{A}$,
\[\textit{v}_2^2(G, \mathcal{I})\leq C\mu(\cup\mathcal{I})+\ee\]
More generally for a sequence $(G_n)_n$ in $V_2^0$ and  a sequence
$(\ee_n)_n$ of positive real numbers, we say that $(G_n)_n$ is
$(C,(\ee_n)_n)$-dominated by $\mu$ if  for every $n\in\nn$ and
every $\mathcal{I}\in\mathcal{A}$, $\textit{v}_2^2(G_n,
\mathcal{I})\leq C\mu(\cup\mathcal{I})+\ee_n$.\end{defn}

\begin{rem}\label{Rem6} Suppose that the sequence $(G_n)_n$ is
$(C,(\ee_n)_n)$-dominated by $\mu$ and $\sum_n\ee_n=\ee<\infty$.
Then by the countable additivity and the monotonicity of $\mu$, it
is easy to see that
 for every  disjoint sequence $(\mathcal{I}_n)_n$ in $\mathcal{A}$, we have
  \begin{equation}\label{add}\sum_n\textit{v}_2^2(G_n,\mathcal{I}_n)\leq C\|\mu\|+\ee\end{equation}
\end{rem}

\begin{prop}\label{pc0}
Let $(G_n)_{n\in\nn}$ be a sequence of functions of $V_2^0$ which
is $(C,(\ee_n)_n)$- dominated by  a positive measure $\mu\in
\mathcal{M}[0,1]$, for some  null sequence $(\ee_n)_n$ of positive
real numbers. Assume also that $(G_n)_n$ is  a seminormalized
sequence in $V_2^0$ and $\lim_n\|G_n\|_{\infty}= 0$.  Then there
is a subsequence of $(G_n)_{n\in\nn}$ equivalent to the usual
basis of $c_0$.
\end{prop}

\begin{proof} Let $0<\dd<\|G_n\|_{V_2}\leq M$. Since   $(\ee_n)_n$ is a null sequence
by passing to a  subsequence  we may assume that $(\ee_n)_n$ is
decreasing and  $\sum_{n}\ee_n=\ee<\infty$. Since $(G_n)_{n}$ is
seminormalized and pointwise convergent to zero, it is weakly null
and so we may also suppose, by passing again to a subsequence,
that it is a basic sequence. Also since
$\lim_n\|G_n\|_{\infty}=0$, by Proposition \ref{pb} and passing to
a further subsequence, we may suppose that $(G_n)_{n}$ is
$(\ee_n)_{n}-$ biorthogonal.

 As $(G_n)_{n}$ is a  basic sequence,
trivially $(G_n)_n$ has a lower $c_0$-estimate. To show that
$(G_n)_{n}$ is dominated by the  $c_0$- basis, let
$(\lambda_k)_{k=1}^n$ be a sequence of scalars and let
$|\lambda_{k_0}|=\max_{1\leq k\leq n}|\lambda_k|$ . Then  by Lemma
\ref{lul2}, we have that
\begin{equation}\label{ff1}\textit{v}_2^2\Big(\sum_{k=1}^n\lambda_k
G_k,\mathcal{I}\Big)\leq
|\lambda_{k_0}|^2\Big(\sum_{k=1}^n\textit{v}_2^2(G_k,\mathcal{I}^{(k)})
+\ee(2M+\ee)\Big),\end{equation}
 for every $\mathcal{I}\in\mathcal{A}$, which by (\ref{add}) gives that
\begin{equation}\label{ff2}\textit{v}_2^2\Big(\sum_{k=1}^n\lambda_k
G_k,\mathcal{I}\Big)\leq
\Big(C\|\mu\|+\ee(2M+1+\ee)\Big)|\lambda_{k_0}|^2\end{equation}
Therefore setting $K=\sqrt{C\|\mu\|+\ee(2M+1+\ee)}$, we conclude
that
\[\Big\|\sum
_{k=1}^n\lambda_k G_k\Big\|_{V_2}\leq K\max_{1\leq k\leq
n}|\lambda_k|\] and the proof of the proposition is complete.
\end{proof}
Remark \ref{Rad} and the above proposition yield the following.
\begin{cor}\label{corrad} The sequence $(R_n)_n$ defined in
(\ref{Radem}) contains a subsequence equivalent to $c_0$ basis.
\end{cor}

We also state the following generalization of Proposition
\ref{pc0} for later use.

\begin{prop}\label{dominate} Let $(G_n)_{n\in\nn}$ be a sequence of $V_2^0$, $(\mu_n)_n$ be a sequence in
$\mathcal{M}^+[0,1]$ and $(\ee_n)_n$ be a null sequence of
positive real numbers with the following properties.
\begin{enumerate}
\item $(G_n)_n$ is a seminormalized sequence in $V_2^0$. \item
$\lim_n\|G_n\|_\infty=0$. \item There is a constant $C>0$ such
that  $G_n$ is $(C,\ee_n)-$dominated by $\mu_n$, for all
$n\in\nn$. \item There is a measure $\mu\in\mathcal{M}^+[0,1]$
such that $(\mu_n)_n$ is norm convergent to $\mu$.
\end{enumerate}  Then there is a subsequence of
$(G_n)_{n\in\nn}$ equivalent to the usual basis of $c_0$.
\end{prop}
\begin{proof} By passing to a subsequence of $(G_n)_n$ we may suppose that
for all $n\in\nn$,
\begin{equation}\label{ff3}\|\mu_n-\mu\|<2^{-n}\end{equation}
 Let $0<\dd<\|G_n\|_{V_2}\leq M$. As in the proof
of Proposition \ref{pc0}, by passing to a further subsequence  we
may also assume that $\sum_{n}\ee_n=\ee<\infty$ and that $(G_n)_n$
is again a basic sequence which in addition is $(\ee_n)_{n}-$
biorthogonal. To show that $(G_n)_n$ is dominated by the
$c_0$-basis, equations (\ref{ff1}) and (\ref{ff3}) give that for
every $\mathcal{I}\in\mathcal{A}$,
\[\begin{split}\textit{v}_2^2\Big(\sum_{k=1}^n\lambda_k
G_k,\mathcal{I}\Big)&\leq
|\lambda_{k_0}|^2\Big(C\sum_{k=1}^n\mu_k(\cup\mathcal{I}^{(k)})
+\ee(2M+1+\ee)\Big)\\&\leq |\lambda_{k_0}|^2(C\|\mu\|
+C+\ee(2M+1+\ee)),\end{split}\]
 which setting  $K=\sqrt{C\|\mu\|+C +\ee(2M+1+\ee)}$, yields that
\[\Big\|\sum _{k=1}^n\lambda_k G_k\Big\|_{V_2}\leq K
\max\{|\lambda_k|:1\leq k\leq n\}\]
\end{proof}

\begin{prop} \label{c011} Let $X$ be a subspace of $V_2^0$ and let $(f_n)_n$ be
a bounded sequence in $X$ pointwise convergent to a function $f\in
(X^{**}\setminus X)\cap C[0,1]$.   Then for every $0<\delta<dist(f,X)$ and
 for every sequence $(\ee_n)_n$ of positive real numbers there exist
a  convex block sequence $(h_n)_n$ of  $(f_n)_n$ such that for all $n<m$ the following  are satisfied.
\begin{enumerate}
\item [(i)] $\delta< \|h_m-h_n\|_{V_2}\leq 2M$, where $M=\sup_n
\|f_n\|_{V_2} $. \item[(ii)] $\|h_m-h_n\|_{\infty}\leq \ee_n$.
\item [(iii)] The function $h_m-h_n$ is $(4,\ee_n)$-dominated by
the measure $\mu_f$.
\end{enumerate}\end{prop}
\begin{proof} Since $f$ is
continuous, we have that
 $\widetilde{osc}_{[0,1]}f=0$ and
$\mu_f^d=0$. The result now follows by Proposition \ref{ldom}.
\end{proof}

\begin{cor} \label{c0emb}Let $X$ be a subspace of $V_2^0$ and  $f\in
(X^{**}\setminus X)\cap C[0,1]$. Then $c_0$ is embedded into $X$.
\end{cor}
\begin{proof} As we have already mentioned  (see Remark \ref{Rem5}), there
is a sequence $(f_n)_n$ in $X$ pointwise convergent to $f$ with
$\|f_n\|_{V_2}\leq \|f\|_{V_2}$. Let $(\ee_n)_n$ be a null
sequence of positive real numbers and let $(h_n)_n$ be a convex
block sequence of $(f_n)_n$ satisfying the properties of
Proposition \ref{c011}. For each $n\in\nn$, let
$G_n=h_{2n}-h_{2n-1}$. By (i)-(iii) of Proposition \ref{c011}, we
have that $(G_n)_n$ is a seminormalized sequence of functions in
$X$, $\lim_n\|G_n\|_\infty=0$ and $(G_n)_n$ is
$(4,(\ee_{2n-1})_n)$-dominated by the measure $\mu_f$. By
Proposition \ref{pc0} the result follows.
\end{proof}

\begin{rem}\label{rem6} Exploiting more carefully Lemma \ref{c011}, we may pass to
an appropriate subsequence of $(h_n)_n$ which is equivalent to the
summing basis. This gives an alternative proof of the known result
 that every $f\in V_2\cap C[0,1]$ is a difference
of bounded semicontinuous functions on the compact metric space
$(B_{(V_2^0)^*},w^*)$. Moreover the converse of Corollary
\ref{c0emb} also holds, that is  $(X^{**}\setminus X)\cap
C[0,1]\neq \emptyset$ if and only if $c_0$ is embedded into $X$
(cf. \cite{AMP}).\end{rem}

\subsection{The embedding of $c_0$ into $X$ when
$\mathcal{M}_{X^{**}}$ is separable.}
\begin{lem}\label{ld} Let $X$ be a subspace of $V_2^0$ and $\mathcal{F}$ be an uncountable
subset of $X^{**}$. If $D_\mathcal{F}=\cup_{f\in\mathcal{F}} D_f$
is countable then for every $\ee>0$ there is an uncountable subset
$\mathcal{F}'\subseteq \mathcal{F}$ such that for every
$f_1,f_2\in \mathcal{F}'$, $\|\mu_{f_1-f_2}^d\|<\ee.$
\end{lem}
\begin{proof} Let
$D_{\mathcal{F}}=\{t_n\}_n$.
 By Proposition  \ref{discr.}, we have that $\|\mu_f^d\|=\sum_{n}\tau_f(t_n)\leq
\|f\|_{V_2}^2$,  for every $f\in \mathcal{F}$. By the definition
of $\tau_f(t)$,  we easily get the following inequalities.
\begin{enumerate}
\item[(a)] For every $f\in V_2$, $\tau_f(t)\leq
4(|f(t^+)|^2+|f(t^+)|^2+|f(t)|^2)$. \item[(b)] For every $f_1,
f_2$ in $V_2$, $\tau_{f_1-f_2}(t)\leq
2\tau_{f_1}(t)+2\tau_{f_2}(t)$.
\end{enumerate}
By passing to an uncountable subset $\mathcal{F}'$ of
$\mathcal{F}$, we may suppose that the following hold.
\begin{enumerate}
\item [(i)] There is $n_0\in\nn$ such that
$\sum_{n>n_0}\tau_f(t_n)< \ee/8$, for all $f\in\mathcal{F}'$.
\item [(ii)] For all $1\leq n\leq n_0$, $j\in\{-,0,+\}$ and
$f_1,f_2\in \mathcal{F}'$, $|(f_1-f_2)(t_n^j)|<\sqrt{\frac{\ee}{24
n_0}}$.
\end{enumerate}
Then by (ii) and (a), we get that for every $f_1,
f_2\in\mathcal{F}'$,
\begin{equation}\label{low1}\sum_{n=1}^{n_0}\tau_{f_1-f_2}(t_n)\leq
4\sum_{j\in\{-,0,+\}}\sum_{n=1}^{n_0}|(f_1-f_2)(t_n^j)|^2<
\ee/2\end{equation} Moreover by (i) and (b),
\begin{equation}\label{low2}\sum_{n>n_0}\tau_{f_1-f_2}(t_n)\leq
2\sum_{n>n_0}\tau_{f_1}(t_n)+2\sum_{n>n_0}\tau_{f_2}(t_n)<
\ee/2\end{equation} Hence, by (\ref{low1}) and (\ref{low2}),
$\|\mu_{f_1-f_2}^d\|=\sum_{n=1}^{n_0}\tau_{f_1-f_2}(t_n)+\sum_{n>n_0}\tau_{f_1-f_2}(t_n)<
\ee$. \end{proof}

\begin{prop}\label{submain}
Let $X$ be a subspace of $V_2^0$ such that $X^*$ is separable,
$X^{**}$ non-separable. If  $\mathcal{M}_{X^{**}}$ is a separable
subset of $\mathcal{M}[0,1]$ then   $c_0$ is embedded into $X$.
\end{prop}

\begin{proof} Let  $\mathcal{F}$ be an uncountable subset of the unit sphere
$S_{X^{**}}$ of $X^{**}$ such that  for all $f_1\neq f_2$ in
$\mathcal{F}$, $\|f_1-f_2\|_{V_2}>3\dd>0$. Since $X$ is separable,
it is easy to see that by  passing  to  a further uncountable
subset, we  may assume that for all  $f_1\neq f_2$ in
$\mathcal{F}$, $dist(f_1-f_2, X)>\dd$. Moreover  since $X^*$ is
separable, by Proposition \ref{P1} the set $D_\mathcal{F}$ is
countable.

Let $(\ee_n)_{n}$ be a sequence of positive real numbers with
$\ee=\sum_{n}\ee_n<\infty$. Using Lemma \ref{ld}, Proposition
\ref{ps} and our assumption that $\mathcal{M}_{X^{**}}$ is a
separable, we easily construct a decreasing sequence
$(\mathcal{F}_n)_{n}$ of uncountable subsets of $\mathcal{F}$ such
that  \begin{equation}\label{R1}\|\mu_{f_1-f_2}^d\|<\ee_n^2,\;\;
\|f_1-f_2\|_\infty<\ee_n\;\;\text{ and}\;\;
\|\mu_{f_1}-\mu_{f_2}\|<\ee_n,\end{equation} for all $n\in\nn$ and
$f_1,f_2\in\mathcal{F}_n$.

Given the above construction, we pick for each $n$, $f_1^n\neq
f_2^n$ in $\mathcal{F}_n$.  Notice that since $(\mathcal{F}_n)_n$
is decreasing,  by (\ref{R1}) we have that
$\|\mu_{f_1^n}-\mu_{f_1^{n+1}}\|<\ee_n$ and so, as
$\sum_n\ee_n<\infty$, the sequence $(\mu_{f_1^n})_n$ is norm
converging to a $\mu_1\in\mathcal{M}[0,1]$. Similarly
$(\mu_{f_2^n})_n$ converges to a $\mu_2\in\mathcal{M}[0,1]$.  Let
$F_n=f_1^n-f_2^n$. Applying for each $n\in\nn$, Proposition
\ref{ldom}  with $F_n$ in place of $f$, we obtain $G_n\in X$
satisfying the following.
\begin{enumerate}
\item [(i)] $\delta<\|G_n\|_{V_2}\leq 2\|F_n\|_{V_2}\leq 4$.
\item[(ii)] $\|G_n\|_{\infty}\leq
4\|F_n\|_{\infty}+\ee_n$.\item[(iii)] For every
$\mathcal{I}\in\mathcal{A}$, $\textit{v}_2^2(G_n,\mathcal{I})\leq
4\mu_{F_n}(\cup\mathcal{I})+32
\|F_n\|_{V_2}\sqrt{\|\mu_{F_n}^d\|}+\ee_n$.
\end{enumerate}
By (\ref{R1}), $\|F_n\|_\infty<\ee_n$ and therefore   (ii) gives
\begin{enumerate}
\item[(iv)] $\lim_n\|G_n\|_\infty=0.$
\end{enumerate}
 Moreover   $\|\mu_{F_n}^d\|<\ee_n^2$,
$\|F_n\|_{V_2}\leq 2$ and setting
$\mu_n=2(\mu_{f_1^n}+\mu_{f_2^n})$, by Proposition \ref{basineq}, $\mu_{F_n}\leq\mu_n$.
Replacing in  (iii) we get that for each $n\in\nn$,
\begin{enumerate}
\item[(v)] For all $\mathcal{I}\in\mathcal{A}$,
$\textit{v}_2^2(G_n,\mathcal{I})\leq
4\mu_n(\cup\mathcal{I})+65\ee_n$,that is  $G_n$ is
$65\ee_n$-dominated by $\mu_n$. \end{enumerate} By (i), (iv), (v)
and since  $(\mu_n)_n$ is norm convergent to $\mu_1+\mu_2$, the
assumptions of Proposition \ref{dominate} are fulfilled and so
there is a subsequence of $(G_n)_n$ equivalent to the $c_0$
basis.\end{proof}

\section{On the embedding of $S^2$ into  subspaces of $V_2^0$.}
This final section includes  the main results of the paper. We
divide this section  into three subsections. In the first
subsection we define the $S^2$- \textit{systems} and we show that
their existence  in a subspace $X$ of $V_2^0$ lead to the
embedding of $S^2$ into $X$. A key ingredient is Lemma \ref{lus2}
which is of independent interest. In the next subsection we define
the $S^2$-generating systems which consist the frame for building
$S^2$-systems. Finally in the third subsection we show that every
subspace $X$ of $V_2^0$ with $\mathcal{M}_{X^{**}}$ non separable
contains an $S^2$ generating system and thus by the preceding
results the space $S^2$ is embedded into $X$. We also show that
$S^2$ is contained into $TF$.

\subsection{$S^2$-Systems.}
In this subsection we will define certain structures
closely related with the embedding of the space $S^2$ into $V_2^0$. We start with the definition of a \textit {system}.
 \begin{defn}\label{system}Let $ M$, $\Lambda$, $\theta$ be positive constants  and $(\ee_n)_{n=0}^\infty$ be  a
sequence of positive real numbers.
An $(\ee_n)_n-$system with constants $(M, \Lambda, \theta)$ is  a family of
the form
\[((G_s,\nu_s,\mathcal{I}_s)_{s\in 2^{<\nn}},(\mathcal{Q}_n)_{n\in\nn}),\]
 where  $(G_s)_{s\in 2^{<\nn}}$ is a family of functions of $V_2^0$,
   $(\nu_s)_{s\in 2^{<\nn}}$ is a  family of positive
Borel measures on $[0,1]$,
 $(\mathcal{I}_s)_{s\in 2^{<\nn}}$ is a family in $\mathcal{A}$, and
  $(\mathcal{Q}_n)_{n\in\nn}$ is an increasing sequence of finite subsets of
  $[0,1]$, satisfying the following properties.
  \begin{enumerate}
  \item [(1)] For every  $s\in 2^{<\nn}$,
   $ \|G_s\|_{V_2}\leq M$ and $\|\nu_s\|\leq \Lambda$.
\item  [(2)] For every $n\geq 0$ and $s\in 2^n$, $\|G_{s}\|_{\infty}\leq \ee_n$.
 \item [(3)] For every $n\geq
0$, the set $\mathcal{Q}_n$ $\ee_n-$ determines the quadratic
variation of $<\{G_s:s\in 2^{n}\}>$. \item [(4)] For every $n\geq
0$, $s\in 2^n$ and every
$\mathcal{I}\in\mathcal{F}(\mathcal{Q}_n)$,
$\textit{v}_2^2(G_s,\mathcal{I})\leq\nu_s(\cup\mathcal{I})+\ee_n$.
 \item[(5)] For every $s\perp t$, $(\mathcal{I}_s,\mathcal{I}_t)$ is a
disjoint pair. \item [(6)] For every $s\in 2^{<\nn}$,
$\textit{v}_2^2(G_s,\mathcal{I}_s)>\theta$.
\end{enumerate}\end{defn}
\begin{rem}\label{remark9} Notice that by property (2), we have that $\lim\|G_{\sigma|n}\|_{\infty}=0$
and therefore by Proposition \ref{treebio}, for every family $(\ee_s)_s$ of positive scalars there is a dyadic subtree $(t_s)_s$ such that
$(G_{t_s})_s$ is $(\ee_s)_s$-biorthogonal.
Moreover by (3) and (4) we have that for every $s\in 2^n$, the function $G_s$ is   $(1,2\ee_n)-$dominated by $\nu_s$.
\end{rem}

\begin{defn}\label{S2} An $(\ee_n)_n-$ $S^2$ system with constants $(M,\Lambda, \theta)$ is an $(\ee_n)_n-$
system $((G_s,\nu_s,\mathcal{I}_s)_{s\in
2^{<\nn}},(\mathcal{Q}_n)_{n}),$ with the same constants  satisfying in addition the following
property. For every  $0\leq n\leq m$, $s\in 2^n$, $t\in 2^m$ with $s\sqsubseteq t$ and
 $\mathcal{I}\in\mathcal{F}(\mathcal{Q}_n)$,
\begin{equation}\label{s2system}|\nu_t(\cup\mathcal{I})-\nu_s(\cup\mathcal{I})|<\ee_n\end{equation}
\end{defn}
\begin{rem} \label{S2rem} Suppose that $(\ee_n)_n$ is a null sequence. Then, as $\mathcal{Q}_n$ is increasing, by (\ref{s2system}) we get that
for every $\sigma\in 2^\nn$ and every $\mathcal{I}\in \mathcal{F}(\cup_n\mathcal{Q}_n)$,
the sequence $(\nu_{\sigma|n}(\cup\mathcal{I}))_n$ is  Cauchy.
\end{rem}
\begin{lem}\label{pc0syst}
Let $((G_s,\nu_s,\mathcal{I}_s)_{s\in
2^{<\nn}},(\mathcal{Q}_n)_{n\in\nn}),$ be an $(\ee_n)_n- S^2$
system with constants $(M,\Lambda, \theta)$. Assume also that $(\ee_n)_n$ is a null sequence.
Then there exist a
family of positive Borel measures $(\nu_\sigma)_{\sigma\in 2^\nn}$
on $[0,1]$ such that   $\sup_\sigma \|\nu_\sigma\|\leq \Lambda$,
and  for all $\sigma\in 2^\nn$, $(G_{\sigma|n})_n$ is $(1,(3\ee_n)_n)-$ dominated by
$\nu_\sigma$.\end{lem}
\begin{proof}
Let $\sigma\in 2^\nn$. Since $(\nu_{\sigma|n})_n$ is a bounded
sequence in $\mathcal{M}[0,1]$, there exist  a subsequence
$(\nu_{\sigma|n})_{n\in L}$ and a positive Borel measure
$\nu_\sigma$ on $[0,1]$, such that $(\nu_{\sigma|n})_{n\in L}$  is
$w^*-$ convergent to $\nu_\sigma$. Fix for the following $k\geq 0$
and a finite family of intervals $\mathcal{I}\in\mathcal{A}$. By
condition (3) of Definition \ref{system}, there exists
$\widetilde{\mathcal{I}}\in\mathcal{F}(\mathcal{Q}_k)$ with
$\widetilde{\mathcal{I}}\preceq \mathcal{I}$ such that
$|\textit{v}_2^2(G_{\sigma|k},\mathcal{I})-\textit{v}_2^2(G_{\sigma|k},\widetilde{\mathcal{I}})|<\ee_k$.
 Moreover  by condition (4) of Definition \ref{system} and (\ref{s2system}),
we get that for every $m>k$,
\begin{equation}\label{w1}\textit{v}_2^2(G_{\sigma|k},\mathcal{I})\leq\textit{v}_2^2(G_{\sigma|k},\widetilde{\mathcal{I}})+\ee_k\leq
\nu_{\sigma|k}(\cup\widetilde{\mathcal{I}})+2\ee_k\leq\nu_{\sigma|m}(\cup\widetilde{\mathcal{I}})+3\ee_k\end{equation}
 By Remark \ref{S2rem}, we have that $\lim_m\nu_{\sigma|m}
(\cup\widetilde{\mathcal{I}})$ exists and so (\ref{w1}) implies that,
\begin{equation}\label{m11}\textit{v}_2^2(G_{\sigma|k},\mathcal{I})\leq\lim_m\nu_{\sigma|m}
(\cup\widetilde{\mathcal{I}})+3\ee_k= \lim_{n\in
L}\nu_{\sigma|n}(\cup\widetilde{\mathcal{I}})+3\ee_k\end{equation}
As the set  $\cup\widetilde{\mathcal{I}}$ is a closed subset of
$[0,1]$ and w$^*-\lim_{n\in L}\nu_{\sigma|n}=\nu_\sigma$, by
Portmateau's theorem (see \cite{Ke}) and the monotonicity of
$\nu_\sigma$, we have that
\begin{equation}\label{w2}\lim_{n\in
L}\nu_{\sigma|n}(\cup\widetilde{\mathcal{I}})\leq\nu_\sigma(\cup\widetilde{\mathcal{I}})\leq\nu_\sigma(\cup\mathcal{I})\end{equation}
By (\ref{m11}) and (\ref{w2}) we conlude that for every $k\geq 0$ and every
$\mathcal{I}\in\mathcal{A}$,
\begin{equation}\label{m22}\textit{v}_2^2(G_{\sigma|k},\mathcal{I})\leq\nu_\sigma(\cup\mathcal{I})+3\ee_k,\end{equation}
that is the sequence $(G_{\sigma|n})_n$ is
$(1,(3\ee_n)_n)$-dominated by the measure $\nu_\sigma$. Finally
since $\nu_\sigma$ is in the w$^*$-closure of $\{\nu_s\}_{s\in
2{<\nn}}$,  $\|\nu_\sigma\|\leq \sup_n\|\nu_{\sigma|n}\|\leq
\Lambda$.
\end{proof}

\begin{rem}\label{chain} Notice that if $\sum_n\ee_n=\epsilon<\infty$, the above lemma yields that
 for  every $\sigma\in 2^\nn$ and every disjoint family $(\mathcal{I}_n)_{n}$ in
$\mathcal{A}$, we have that
\begin{equation}\label{diseq}\sum_n\textit{v}_2^2(G_{\sigma|n},\mathcal{I}_n)\leq \Lambda+3\epsilon\end{equation}

The next lemma concerns an inequality for tree families of
positive numbers which is critical for the embedding of $S^2$ into
subspaces of $V_2^0$ and could be useful elsewhere.
\end{rem}
\begin{lem}\label{lus2}
Let $(\alpha_s)_{s\in 2^{<\nn}}$ , $(\lambda_s)_{s\in 2^{<\nn}}$
be two families of non negative real numbers and let $n\geq 0$.
Then there exists a maximal  antichain $A$ of $ 2^{\leqslant n}$
and a family of branches $(b_t)_{t\in A}$ of $2^{\leqslant n}$
such that
 $\sum_{s\in 2^{\leqslant n}}\lambda_s\alpha_s\leq \sum_{t\in
A}(\sum_{s\in b_t}\alpha_s)\lambda_t$ and $t\in b_t$, for all
$t\in A$. Therefore  if $\sum_{n=1}^{\infty}\alpha_{\sigma|n}\leq
C$,  for all $\sigma\in 2^{\nn}$, then for each  $n\geq 0$ there
 is an antichain $A$ of $2^{\leqslant n}$ such that
$\sum_{s\in 2^{\leqslant n}}\lambda_s\alpha_s\leq C\sum_{s\in A}\lambda_s$.
\end{lem}

\begin{proof} We shall use induction on $n\geq 0$. The lemma  trivially holds  for $n=0$.
Assuming that  it is true for some $n$, we show the $n+1$ case.
For each $j\in\{0,1\}$, let $\mathcal{D}_j=\{t\in
2^{n+1}:\;t(1)=j\}$. Then  $\mathcal{D}_j$ is order isomorphic to
$2^{\leqslant n}$ and so  by our inductive assumption there is an
antichain $A_j\subseteq\mathcal{D}_j$, and a family of branches
$\{b^j_t:t\in A_j\}\subseteq\mathcal{D}_j$ with $t\in b_t^j$, for
each $t\in A_j$ and $\sum_{s\in\mathcal{D}_j}\lambda_s\alpha_s\leq
\sum_{t\in A_j}(\sum_{s\in b^j_t}\alpha_s)\lambda_t$. Hence we
easily get that
\begin{equation} \label{tr1}
\sum_{s\in 2^{\leqslant n+1}}\lambda_s\alpha_s=\lambda_{\emptyset}\alpha_{\emptyset}+
\sum_{t\in \mathcal{D}_0\cup \mathcal{D}_1}\lambda_t\alpha_t\leq\lambda_{\emptyset}\alpha_{\emptyset}+
\sum_{t\in A_0}(\sum_{s\in b^0_t}\alpha_s)\lambda_t+\sum_{t\in A_1}(\sum_{s\in b^1_t}\alpha_s)\lambda_t\end{equation}
We distinguish two cases.\\
\textbf{Case 1}: $\lambda_{\emptyset}\leq\sum_{t\in
A_0}\lambda_t+\sum_{t\in A_1}\lambda_t.$ Then let $A=A_0\cup A_1$
and $b_t=b_t^j\cup\{\emptyset\}$, for each $t\in A_j$. Obviously
$A$ is a maximal antichain in $2^{\leqslant n+1}$ and  $\{b_t:
\;t\in A\}$ is a family of branches with $t\in b_t$ for each $t\in
A$. Moreover as  $\lambda_{\emptyset}\leq\sum_{t\in A}\lambda_t$,
by (\ref{tr1}) we obtain that
 \[\sum_{s\in 2^{\leqslant n+1}}\lambda_s\alpha_s\leq\sum_{t\in A}\Big(\sum_{s\in b_t}\alpha_s\Big)\lambda_t.\]
\textbf{Case 2}:
$\sum_{t\in A_0}\lambda_t+\sum_{t\in A_1}\lambda_t<\lambda_{\emptyset}.$ Then we set
 $A=\{\emptyset\}$ and  $b_\emptyset=\{\emptyset\}\cup b_{t_0}^{j_0} $ where
$\sum_{s\in b_{t_0}^{j_0}}\alpha_s=\max\bigcup_{j=0}^1\{\sum_{s\in
b_{t}^{j}}\alpha_s:t\in  A_j\}$. By  (\ref{tr1}), we get that
\[\sum_{s\in 2^{\leqslant n+1}}\lambda_s\alpha_s\leq
\lambda_{\emptyset}\Big(\alpha_{\emptyset}
+\sum_{s\in b^{j_0}_{t_0}}\alpha_s\Big)\leq (\sum_{s\in b_\emptyset}\alpha_s)\lambda_\emptyset\]\end{proof}

\begin{prop}\label{corgen1}  Let $((G_s,\nu_s,\mathcal{I}_s)_{s\in 2^{<\nn}},(\mathcal{Q}_n)_{n}),$ be  an
 $(\ee_n)_n-S^2$   system with constants $(M,\Lambda, \theta)$. Suppose that $(\ee_n)_n$ is a
 summable sequence of positive real numbers.
 Then there is a dyadic
 subtree $(t_s)_{s\in 2^n}$ of $2^{<\nn}$ such that $(G_{t_s})_{s\in 2^{<\nn}}$ is equivalent to the $S^2$-basis.
\end{prop}

\begin{proof}
Let $(\ee_s)_{s\in 2^{<\nn}}$ be a  family of positive real
numbers such that $\sum_s\ee_s=\ee<\infty$ and
$\theta-(\ee+2M)\ee>0$. As we have already mentioned (see  Remark
\ref{remark9}), there is a dyadic subtree $(t_s)_{s\in 2^{<\nn}}$
of $2^{<\nn}$ such that $(G_{t_s})_{s\in 2^{<\nn}}$ is
$(\ee_s)_{s\in 2^{<\nn}}-$ biorthogonal. We will show that
$(G_{t_s})_{s\in 2^{<\nn}}$ is equivalent to the $S^2$-basis. To
this end, fix
  a sequence of real numbers
 $(\lambda_s)_{|s|\leq n}$.

 First we show the upper
 $S^2$-estimate. Let
   $\mathcal{I}\in\mathcal{A}$.  By  Lemma \ref{lul2} we have  that
\begin{equation}\label{sq1}\textit{v}_2^2\Big(\sum_{|s|\leq n}\lambda_s G_{t_s},\mathcal{I}\Big)\leq
\sum_{|s|\leq
n}|\lambda_s|^2\textit{v}_2^2(G_{t_s},\mathcal{I}^{(t_s)})
+\max_{|s|\leq n}|\lambda_s|^2(2M+\ee)\ee\end{equation}
  By Remark \ref{chain},     we have that
$\sum_n\textit{v}_2^2(G_{t_{\sigma|n}},\mathcal{I}^{(t_{\sigma|n})})\leq
\Lambda+3\epsilon,$ where  $\epsilon=\sum_n\ee_n$. Hence by Lemma
\ref{lus2}, (with $\textit{v}_2^2(G_{t_s},\mathcal{I}^{(t_s)})$
and $|\lambda_s|^2 $ in place of $\alpha_s$ and $|\lambda_s|$
respectively), we obtain an antichain $A\subseteq 2^{\leqslant n}$
such that
\begin{equation}\label{sq2}\sum_{|s|\leq n}|\lambda_s|^2\textit{v}_2^2(G_{t_s},\mathcal{I}^{(t_s)})
\leq (\Lambda+3\epsilon)\sum_{s\in A}|\lambda_s|^2\end{equation}
By (\ref{sq1}) and (\ref{sq2}), we get that
\[\begin{split}\textit{v}_2^2\Big(\sum_{|s|\leq n}\lambda_s G_{t_s},\mathcal{I}\Big)&\leq
(\Lambda+3\epsilon)\sum_{s\in A}|\lambda_s|^2 +\max_{|s|\leq
n}|\lambda_s|^2(2M+\ee)\ee
\\&\leq \Big(\Lambda+3\epsilon+(2M+\ee)\ee\Big)\Big\|\sum_{|s|\leq n}\lambda_s e_s\Big\|^2_{S^2},\end{split}\]
This yields that there is $C>0$ such that  $\|\sum_{|s|\leq
n}\lambda_s G_{t_s}\|_{V_2}\leq C \|\sum_{|s|\leq n}\lambda_s
e_s\|_{S^2}$.

We pass now to show the lower $S^2$- estimate.
 Let $A$ be an antichain
of $2^{\leqslant n}$ such that
\begin{equation}\label{ant0}\big\|\sum_{|s|\leq n}\lambda_se_s\Big\|_{S^2}=
\Big(\sum_{s\in A}|\lambda_s|^2\Big)^{1/2}\end{equation}  Since
$(\mathcal{I}_{t_s})_{s\in A}$ is a disjoint family, we get that
$\mathcal{I}=\bigcup_{s\in A}\mathcal{I}_{t_s}\in\mathcal{A}$. By
Lemma \ref{lul2} and (\ref{ant0}), we have that
\begin{equation}\label{ant33}\begin{split}\textit{v}_2^2\Big(\sum_{|s|\leq n}\lambda_s G_{t_s},\mathcal{I}\Big)&\geq
\sum_{s\in
A}|\lambda_s|^2\textit{v}_2^2(G_{t_s},\mathcal{I}^{(t_s)})-
\max_{|s|\leq n}|\lambda_s|^2\Big(\sum_{|s|\leq
n}\ee_s\Big)2M\\&\geq\sum_{s\in
A}|\lambda_s|^2\textit{v}_2^2(G_{t_s},\mathcal{I}^{(t_s)})-\Big(\sum_{s\in
A}|\lambda_s|^2\Big)2M\ee\end{split}\end{equation} By the
properties of the $S^2$-system, we have that  for every $s\in A$,
\[\theta<\textit{v}_2^2(G_{t_s},\mathcal{I}_{t_s})\leq\textit{v}_2^2(G_{t_s},\mathcal{I})
=\textit{v}_2^2(G_{t_s},\mathcal{I}^{(t_s)})+\textit{v}_2^2(G_{t_s},\mathcal{I}\setminus
\mathcal{I}^{(t_s)})\leq
\textit{v}_2^2(G_{t_s},\mathcal{I}^{(t_s)})+\ee^2\] and therefore
$\textit{v}_2^2(G_{t_s},\mathcal{I}^{(t_s)})\geq \theta-\ee^2$.
Hence by (\ref{ant33}), we obtain that
\[\textit{v}_2^2\Big(\sum_{|s|\leq n}\lambda_s G_{t_s},\mathcal{I}\Big)\geq
\Big(\theta-(\ee+2M) \ee\Big)\Big(\sum_{s\in
A}|\lambda_s|^2\Big),\] which gives that there is $c>0$ such that
$\|\sum_{|s|\leq n}\lambda_s
 G_{t_s}\|_{V_2}\geq c\;
\|\sum_{|s|\leq n}\lambda_s e_s\|_{S^2}$.
\end{proof}

\subsection{$S^2$-generating systems.}
\begin{defn}\label{pres2} An $(\ee, (\ee_n)_n)-$ $S^2$ generating system with constants
$(M,\Lambda, \theta)$ is an $(\ee_n)_n-$ system
$((H_s,\mu_s,\mathcal{J}_s)_{s\in 2^{<\nn}},(\mathcal{P}_n)_{n})$,
with the  same constants  satisfying in addition the following
properties.
\begin{enumerate}
\item[(i)] For every  $m>n\geq 0$, $s\in 2^n$ and
$\{s_0,s_1\}\subseteq 2^m$ such that $s^\smallfrown 0\sqsubseteq
s_0$ and $s^\smallfrown 1\sqsubseteq s_1$ and every
$\mathcal{I}\in\mathcal{F}(\mathcal{P}_n)$, we have that
\begin{equation}\label{s2gen}|\frac{\mu_{s_0}+\mu_{s_1}}{2}(\cup\mathcal{I})-\mu_s(\cup\mathcal{I})|<\ee_n\end{equation}
\item[(ii)] For every $n\geq 1$, the sequence $(H_s)_{s\in 2^n}$
is $(\ee_s)_{s\in 2^n}$-biorthogonal,
 where $\ee=\sum_{s\in 2^n}\ee_s>0$.
\end{enumerate}
\end{defn}
In the following we present a partition of $2^{<\nn}$ in
continuoum many almost disjoint  subtrees. This partition is
induced by the canonical bijection between $2^{<\nn}$ and
$2^{<\nn}\times 2^{<\nn}$.
\begin{defn} \label{lev}Let $s\in 2^{<\nn}$. If $s=\emptyset$ then let $L_\emptyset=\{\emptyset\}$.
If $\emptyset\neq s=(s(1),...,s(n))$ let $L_s=\{t\in
2^{2n}:\;t(2i)=s(i),\;\text{for all}\;1\leq i\leq n\}$.
\end{defn}
\begin{rem}\label{remalmost disjoint} It is easy to see that for each $\sigma\in 2^\nn$,
$T_\sigma=\cup_n L_{\sigma|n}$, is a dyadic subtree of $2^{<\nn}$
and  $(T_\sigma)_{\sigma\in 2^\nn}$ is an  almost disjoint family
and hence their bodies $([T_\sigma])_{\sigma\in 2^\nn}$ are
disjoint.
\end{rem}
 The following properties of $(L_s)_{s\in 2^{<\nn}}$ are easily established.
\begin{enumerate}
\item[(L1)] For all $n\geq 0$ and $s\in 2^{n}$, $L_s\subseteq
2^{2n}$ and $|L_s|=2^n$, where $|L_s|$ is the cardinality of
$L_s$. \item[(L2)] For $s_1\sqsubseteq s_2$, $L_{s_1}=L_{s_2}|
n_1=\{t| n_1:\;t\in L_{s_2}\},$ where $n_1=2|s_1|$. \item[(L3)]
For $s_1\perp s_2$, $L_{s_1}\perp L_{s_2}$. \item[(L4)] For all
$n\geq 0$, $2^{2n}=\cup_{s\in 2^n}L_s$.
\end{enumerate}

Given an an  $(\ee,(\ee_n)_n)-S^2$ generating  system
$((H_s,\mu_s,\mathcal{J}_s)_{s\in 2^{<\nn}},(\mathcal{P}_n)_{n})$,
with constants $(M,\Lambda, \theta)$, we set $G_s=
2^{-n/2}\sum_{t\in L_s}H_t$, $ \nu_s= 2^{-n}\sum_{t\in L_s}\mu_t$,
$ \mathcal{I}_s=\bigcup_{t\in L_s}\mathcal{J}_t$,
$\mathcal{Q}_n=\mathcal{P}_{2n}$ and $\ee_n'=\frac{\theta}{
2^{n/2}}+2^{n/2}\ee_{2n}$.

\begin{prop}\label{gen1} If
$\theta'=\theta-(2M+\ee)\ee>0$ then the system
$((G_s,\nu_s,\mathcal{I}_s)_{s\in 2^{\nn}},(\mathcal{Q}_n)_{n})$
is an $(\ee'_n)_n-S^2$
   system with constants $(M+\ee,\Lambda, \theta')$. \end{prop}

\begin{proof}  Let $n\geq 1$ and $s\in 2^n$. Since $L_s\subseteq 2^{2n}$,
 $|L_s|=2^n$ and $(H_t)_{t\in 2^{2n}}$ is $(\ee_t)_{t\in 2^{2n}}$-biorthogonal,
with $\sum_{t\in 2^{2n}}\ee_t=\ee<1$, by Lemma \ref{lul2}, we get
that for every $\mathcal{I}\in\mathcal{A}$,
\[\textit{v}_2^2(G_s,\mathcal{I})< |L_s|^{-1} \sum_{t\in L_s}
 \textit{v}_2^2(H_t,\mathcal{I}^{(t)})+ 2^{-n}(2M+\ee)\ee\leq (M+\ee)^2\]
 and so $\|G_s\|_{V_2^0}\leq M+\ee$.

 Moreover  $\|G_s\|_{\infty}\leq\sqrt{ |L_s|} \ee_{2n}= \sqrt{2^{n}}\ee_{2n}<\ee_n'.$
We  show now that the quadratic variation of $<\{G_s\}_{s\in
2^n}>$ is   $\ee_n'$-determined by $\mathcal{Q}_n$. Let
$\mathcal{I}\in\mathcal{A}$. Since $\mathcal{P}_{2n}$, $\ee_{2n}$-
determines the quadratic variation of $<\{H_t\}_{t\in 2^{2n}}>$,
there is $\widetilde{\mathcal{I}}\preceq \mathcal{I}$  in
$\mathcal{F}(\mathcal{P}_{2n})$ such  that for all $(\mu_t)_{t\in
2^{2n}}$,
\[|\textit{v}_2^2(\sum_{t\in 2^{2n}}\mu_t H_t,\mathcal{I})-\textit{v}_2^2(\sum_{t\in 2^{2n}}\mu_t H_t,
\widetilde{\mathcal{I}})|\leq( \sum_{t\in
2^{2n}}|\mu_t|^2)\ee_{2n}\]  Notice that for every sequence of
scalars  $(\lambda_s)_{s\in 2^n}$,
 $\sum_{s\in 2^n}\lambda_s G_s=\sum_{t\in 2^{2n}}\mu_t H_t,$ where for
 each $t\in 2^{2n}$, $\mu_t=|L_s|^{-1/2}\lambda_s$, for $t\in
 L_s$. Therefore
  \[\begin{split}|\textit{v}_2^2(\sum_{s\in 2^n}\lambda_s G_s, \mathcal{I})-\textit{v}_2^2(\sum_{s\in 2^n}\lambda_s G_s,
   \widetilde{ \mathcal{I}})|&=
 |\textit{v}_2^2(\sum_{t\in 2^{2n}}\mu_t H_t,\mathcal{I})-\textit{v}_2^2(\sum_{t\in 2^{2n}}\mu_t H_t,\widetilde{\mathcal{I}})|
 \\&\leq( \sum_{t\in 2^{2n}}|\mu_t|^2)\ee_{2n}
 \leq(\sum_{s\in 2^n}|\lambda_s|^2)\ee_n'\end{split}\]
We proceed to show property (4) of Definition \ref{system}. Let
$\mathcal{I}\in \mathcal{F}(\mathcal{Q}_n)$. Then $\mathcal{I}\in
\mathcal{F}(\mathcal{P}_{2n})$, $\mathcal{I}=\cup_{t\in
2^{2n}}\mathcal{I}^{(t)}$ and for all $t\in 2^{2n}$, $
\textit{v}_2^2(H_t,\mathcal{I}^{(t)})\leq \mu_t(\cup
\mathcal{I}^{(t)})+\ee_{2n}$. Notice also that for every $t\in
L_s$,  $\mu_t\leq |L_s|\nu_s$. Hence, by Lemma  \ref{lul2}, we get
that \[\begin{split}
\textit{v}_2^2(G_s,\mathcal{I})& 
                            \leq|L_s|^{-1}\sum_{t\in L_s}\mu_t(\cup \mathcal{I}^{(t)})+\ee_{2n}+2^{-n}(2M+\ee)\ee\\
                            &\leq\sum_{t\in L_s}\nu_s(\cup
                            \mathcal{I}^{(t)})+\ee_{2n}+3\theta/4\sqrt{2^n}
                            \leq\nu_s(\cup \mathcal{I})+\ee_n'\end{split}\]
We  show now   (\ref{s2system}) of Definition \ref{S2}. Let $m\geq
1$ and $t=(t(1),...,t(m))\in 2^m$. Let $0\leq n<m$ and $s= t|n$.
  For each $w\in L_s$, let $w_0=w^\smallfrown 0^\smallfrown t(n+1)$ and $w_1=w^\smallfrown 1^\smallfrown t(n+1)$.
 Let also $u=(u(1),...,u(m-n-1))$ where $u(i)=t(n+1+i)$, for all $1\leq i\leq m-n-1$.
 Then $L_t=\bigcup_{w\in L_{s}}(w_0^\smallfrown L_u\cup w_1^\smallfrown L_u)$ and so
  $\nu_t=2^{-m}\sum_{w\in L_s}\sum_{v\in L_u}(\mu_{w_0^\smallfrown v}+\mu_{w_1^\smallfrown v})$. Therefore,
 \[\begin{split}\nu_t-\nu_s &=2^{-m}\sum_{w\in L_s}\sum_{v\in L_u}(\mu_{w_0^\smallfrown v}+
 \mu_{w_1^\smallfrown v})-2^{-n}\sum_{w\in L_s}\mu_w\\
&= 2^{-m}\sum_{w\in L_s}\sum_{v\in L_u}(\mu_{w_0^\smallfrown
v}+\mu_{w_1^\smallfrown v}-2\mu_w)\end{split}\] Moreover by
(\ref{s2gen}), for every
$\mathcal{I}\in\mathcal{F}(\mathcal{Q}_n)=\mathcal{F}(\mathcal{P}_{2n})$,
and all  $v\in L_u$, we have that \[|(\mu_{w_0^\smallfrown
v}+\mu_{w_1^\smallfrown v}-2\mu_w)(\cup \mathcal{I})|\leq
2\ee_{2n}\]  Hence $|\nu_t(\cup \mathcal{I})-\nu_s (\cup
\mathcal{I})| \leq 2^{-m} 2^n 2^{m-n-1}2\ee_{2n}=\ee_{2n}\leq
\ee_n'$.

That the pair $(\mathcal{I}_s, \mathcal{I}_t)$ is disjoint,  is
straightforward by properties (5) of Definition \ref{system} and
(L3) of Definition \ref{lev}. Finally,
$\textit{v}_2^2(H_t,\mathcal{J}_t)>\theta$, implies  that
$\textit{v}_2^2(H_t,\mathcal{I}_s^{(t)})>\theta-\ee^2$, for all
$t\in 2^{<\nn}$. Hence by Lemma \ref{lul2}, we obtain that
\[\textit{v}_2^2(G_s,\mathcal{I}_s)\geq \sum_{t\in
L_s}|L_s|^{-1}\textit{v}_2^2(H_t,\mathcal{I}_s^{(t)})-2^{-n}2M\ee\geq\theta-(\ee+2M)\ee=\theta'\]
\end{proof}

\subsection{The embedding of $S^2$ into $X$ when $\mathcal{M}_{X^{**}}$ is non-separable.}

\begin{lem} \label{fact1}Let $\{\mu_\xi\}_{\xi<\omega_1}$ be a non-separable  subset of $\mathcal{M}^+[0,1]$.
Then there are an uncountable subset $\Gamma$  of $\omega_1$ such
that for every $\xi\in \Gamma$,
 $\mu_\xi=\lambda_\xi+\tau_\xi$  where $\lambda_\xi,\tau_\xi$ are positive Borel measures on $[0,1]$
  satisfying the following properties.
\begin{enumerate}
\item For all $\xi\in \Gamma$,  $\lambda_\xi\perp\tau_\xi$ and
$\|\tau_\xi\|>0$. \item For all $\zeta<\xi$ in $\Gamma$,
$\mu_\zeta\perp \tau_\xi$.
\end{enumerate}
\end{lem}
\begin{proof} We may suppose that for some $\delta>0$,
 $\|\mu_\xi-\mu_\zeta\|>\delta$, for all $0\leq \zeta<\xi<\omega_1$. By transfinite induction we
 construct a strictly increasing sequence $(\xi_\alpha)_{\alpha<\omega_1}$ in $\omega_1$ such that
 for each $\alpha<\omega_1$, $\mu_{\xi_\alpha}=\lambda_{\xi_\alpha}+\tau_{\xi_\alpha}$, with
  $\lambda_{\xi_\alpha}\perp \tau_{\xi_\alpha}$, $\|\tau_{\xi_\alpha}\|>0$ and  $\tau_{\xi_\alpha}\perp \mu_{\xi_\beta}$ for all $\beta<\alpha$.
  The general inductive step of the construction has as follows. Suppose that for some $\alpha<\omega_1$,
$(\xi_\beta)_{\beta<\alpha}$ has been defined. Let $(\beta_n)_n$ an enumeration of $\alpha$ and set
 \[\zeta_\alpha=\sup_n\xi_{\beta_n},\;\;
 \nu_\alpha=\sum_n\mu_{\xi_{\beta_n}}/2^n\;\text{and}\;N_\alpha=\{\xi<\omega_1:\zeta_\alpha<\xi\;\text{and}\;\mu_\xi<<\nu_\alpha\}\]
By Radon- Nikodym theorem, $\{\mu_\xi\}_{\xi\in N_\alpha}$ is
isometrically contained in $L_1([0,1],\nu_\alpha)$ and therefore
it is norm separable. Since we have assumed that
$\|\mu_\xi-\mu_\zeta\|>\delta$, for all $0\leq
\zeta<\xi<\omega_1$, we get that $N_\alpha$ is countable. Hence we
can choose $\xi_\alpha>\sup N_\alpha$. Let
$\mu_{\xi_\alpha}=\lambda_{\xi_\alpha}+\tau_{\xi_\alpha}$ be the
Lebesgue analysis of $\mu_{\xi_\alpha}$ where
$\lambda_{\xi_\alpha}<<\nu_{\alpha} $ and $\tau_{\xi_\alpha}\perp
\nu_{\alpha}$. By the definition of $\nu_{\alpha}$
 and $\xi_\alpha$, we have that $\|\tau_{\xi_\alpha}\|>0$, $\tau_{\xi_\alpha}\perp\mu_{\xi_\beta}$,
 for all $\beta<\alpha$ and the inductive step
 of the construction has
been  completed.\end{proof}

\begin{lem} \label{fact2} Let $\{\tau_\xi\}_{\xi<\omega_1}$ be an uncountable family
of pairwise singular positive Borel measures on $[0,1]$.
 Then for  every finite family  $(\Gamma_i)_{i=1}^k$ of  pairwise disjoint uncountable subsets of $\omega_1$
and every $\ee>0$  there exist a family  $(\Gamma'_i)_{i=1}^k$  with  $\Gamma'_i$  uncountable subset
  of $\Gamma_i$  and  a family $(U_i)_{i=1}^k$ of open and pairwise disjoint subsets of  $[0,1]$ such that
$\tau_\xi([0,1]\setminus U_i)<\ee,$
   for all $1\leq i\leq k$ and  $\xi\in \Gamma_i'$.
\end{lem}
\begin{proof} For every $\alpha<\omega_1$, we choose $(\xi_i^\alpha)_{i=1}^k\in \prod_{i=1}^k{\Gamma_i}$
 such that for every $\alpha\neq \beta$ in $\omega_1$ and every $1\leq i\leq k$, $\xi_i^\alpha\neq \xi_i^\beta$.
For each $0\leq \alpha<\omega_1$  the k-tuple
$(\tau_{\xi^i_\alpha})_{i=1}^k$ consists of pairwise singular
measures and so we may choose a k-tuple $(U^\alpha_i)_{i=1}^k$ of
open subsets of $[0,1]$ with the following properties: (a) For
each $i$, $\tau_{\xi^\alpha_i}([0,1]\setminus U^\alpha_i)<\ee$,
(b) For all $i\neq j$, $U_i^\alpha\cap U_j^\alpha=\emptyset$ and
(c) For each $i$, $U_i^\alpha$ is a finite union of open in
$[0,1]$
 intervals with rational endpoints.

Since the family of all finite unions of  open  intervals with
rational endpoints
 is countable, there is a $k$-tuple $(U_i)_i$ and an uncountable subset $\Gamma$ of $\omega_1$,
  such that for all $1\leq i\leq k$ and all $\alpha\in \Gamma$,
 $U_i^\alpha=U_i$. For each $1\leq i\leq k$, set $\Gamma_i'=\{\xi_i^\alpha:\;\alpha\in \Gamma\}$.  Then
 for each $1\leq i\leq k$ and all  $\xi\in \Gamma_i'$, $\tau_{\xi}([0,1]\setminus U_i)<\ee$.\end{proof}

\begin{lem}\label{lemind}  Let $X$ be a subspace of $V_2^0$ and
suppose that $ X^{**}$ contains an uncountable family
$\mathcal{F}$ such that $D_\mathcal{F}=\cup_{f\in\mathcal{F}}D_f$
is countable and
$\mathcal{M}_{\mathcal{F}}=\{\mu_f\}_{f\in\mathcal{F}}$ is
non-separable. Then there are constants $(M, \Lambda,\theta)$ such
that for every $\ee>0$ and every sequence $(\ee_n)_n$ of positive
scalars there is an $(\ee,(\ee_n)_n)$- $S^2$ generating system
$((H_s,\mu_s,\mathcal{J}_s)_{s\in 2^{<\nn}},(\mathcal{P}_n)_{n})$
 with constants $(M, \Lambda, \theta)$ and $H_s\in X$, for all $s\in 2^{<\nn}$.
\end{lem}

\begin{proof}  Since for all $f\in V_2$ and  $\lambda\in\rr$,
$\mu_{\lambda f}=\lambda^2\mu_f$, we may assume  that
$\mathcal{F}\subseteq S_{X^{**}}$.
 By Lemma \ref{fact1},  there is  a non-separable  subset   $\{\mu_\xi\}_{\xi<\omega_1}$
 of $\mathcal{M}_{\mathcal{F}}$  such that
 for all $0\leq \xi<\omega_1$,
$\mu_\xi=\lambda_\xi+\tau_\xi$,  $\lambda_\xi\perp \tau_\xi$ and
for all $\zeta<\xi$, $\mu_{\zeta}\perp \tau_{\xi}$. By passing to
a further uncountable subset, we may also assume that there is
$\theta_0>0$ such that $\|\tau_\xi\|>\theta_0$.
 We fix $\ee>0$ and a sequence
$(\ee_n)_n$ of positive real numbers. We will construct the
following objects:
\begin{enumerate}
\item[(1)] A Cantor scheme $(\Gamma_s)_{s}$ of uncountable subsets
of $\omega_1$  (that is for all $s\in 2^{<\nn}$,
$\Gamma_{s^\smallfrown 0}\cup\Gamma_{s^\smallfrown 1}\subseteq
\Gamma_{s}$ and  $\Gamma_{s^\smallfrown
0}\cap\Gamma_{s^\smallfrown 1}=\emptyset$), \item[(2)] A family
$((\xi_s^0,\xi_s^1))_s$ of pairs  with $\xi_s^0<\xi_s^1$ in
$\Gamma_s $, for all $s\in 2^{<\nn}$. \item [(3)] A Cantor scheme
of open subsets $(V_s)_s$ of $[0,1]$, \item [(4)] A family of
functions $(H_s)_s$ in $X$, \item[(5)] An increasing sequence
$(\mathcal{P}_n)_n$ of  finite subsets of $[0,1]\setminus
D_{\mathcal{F}}$, and \item [(6)] A family $(\mathcal{J}_s)_s$ in
$\mathcal{A}$,
\end{enumerate}
such that the following are satisfied.
\begin{enumerate}
\item [(i)] For every $\xi\in\Gamma_s$, $\tau_\xi(V_s)>\theta_0/2$
and $\tau_\xi([0,1]\setminus
V_s)<\big(\sum_{i=0}^{|s|}2^{-(i+2)}\big)\theta_0$. \item[(ii)]
The measures $\mu_{\xi_{s}^0}$ and $\mu_{\xi_{s}^1}$ are
$w^*$-condensation points of $\{\mu_\xi\}_{\xi\in \Gamma_s}$.
\item [(iii)] For every $n\geq 1$, $s\in 2^n$, $\xi\in \Gamma_s$
and $\mathcal{I}\in\mathcal{F}(\mathcal{P}_{n-1})$,
\[(\mu_\xi-\mu_{\xi^{s(n)}_{s^-}})(\cup\mathcal{I})|<\frac{\ee_{n-1}}{16},\;\text{where}\;\; s^-=(s(1),...,s(n-1))\]
 \item [(iv)] $\|H_s\|_{V_2}\leq 2$,  $\cup
\mathcal{J}_s\subseteq V_s$ and
$\textit{v}_2^2(H_s,\mathcal{J}_s)>\theta_0/2$. \item[(v)] For
every $s\in 2^n$, $\|H_s\|_\infty\leq \ee_{n}$ and
$\textit{v}_2^2(H_s,\mathcal{I})\leq
8(\mu_{\xi^0_{s}}+\mu_{\xi^1_s})(\cup\mathcal{I})+\ee_{n}$, for
all $\mathcal{I}\in \mathcal{F}(\mathcal{P}_{n})$.
 \item [(vi)] If $(s_i)_{i=1}^{2^n}$
is the lexicographical enumeration of $\{0,1\}^n$, then
$(H_{s_i})_{i=1}^{2^n}$ is $(\ee/2^i)_{i=1}^{2^n}$-biorthogonal.
\item[(vii)] The set $\mathcal{P}_n$ $\ee_n-$determines the
quadratic variation of $<\{H_s\}_{s\in 2^{<\nn}}>$.
\end{enumerate}
Given the above construction,   we set
$\mu_s=8(\mu_{\xi_s^0}+\mu_{\xi_s^1})$ and we claim that the
family $((H_s,\mu_s,\mathcal{J}_s)_{s\in
2^{<\nn}},(\mathcal{P}_n)_{n})$ is an $(\ee,(\ee_n)_n)$- $S^2$
generating system with constants $(M,\Lambda,\theta)$, where
$M=2$, $\Lambda=16$ and $\theta=\theta_0/2$. We  only verify
condition (\ref{s2gen}) of Definition \ref{pres2}, since the other
conditions are immediate. So let $n<m$, $s\in 2^n$ and $s_0,s_1\in
2^m$, with $s^\smallfrown 0\sqsubseteq s_0$ and $s^\smallfrown
1\sqsubseteq s_1$. Then $\Gamma_{s_0}\subseteq
\Gamma_{s^\smallfrown 0}$, $\Gamma_{s_1}\subseteq
\Gamma_{s^\smallfrown 1}$, and so by (iii), for all
$\mathcal{I}\in\mathcal{F}(\mathcal{P}_{n})$ and $j\in\{0,1\}$, we
get that
\begin{equation}\label{measure1}\max\{|(\mu_{\xi_{s_0}^{j}}-\mu_{\xi_s^{0}})(\cup
\mathcal{I})|,
 |(\mu_{\xi_{s_1}^j}-\mu_{\xi_s^{1}})(\cup \mathcal{I})|\}\leq
 \frac{\ee_{n}}{16}\end{equation}
Since
\[\begin{split}|\frac{\mu_{s_0}+\mu_{s_1}}{2}-\mu_s| &
\leq 4
(|\mu_{\xi_{s_0}^0}-\mu_{\xi_{s}^0}|+|\mu_{\xi_{s_0}^1}-\mu_{\xi_{s}^0}|+
|\mu_{\xi_{s_1}^0}-\mu_{\xi_{s}^1}|+|\mu_{\xi_{s_1}^1}-\mu_{\xi_s^1}|,)
\end{split}\]
 by (\ref{measure1}), we have that   for all $\mathcal{I}\in\mathcal{F}(\mathcal{P}_n)$,
$|(\frac{\mu_{s_0}+\mu_{s_1}}{2}-\mu_s)(\cup\mathcal{I})|\leq\ee_n$.

 We present now the general
inductive step of the construction. Let us suppose that the
construction has been carried out for  all $s\in 2^{<n}$. For
every $s=(s(1),...,s(n))\in 2^{n}$, we define
 \begin{equation}\label{eqcon1} \Gamma_{s}^{(1)}=\{\xi\in\Gamma_{s^-}:
 \;\forall\mathcal{I}\in \mathcal{F}(\mathcal{P}_{n-1}),
|(\mu_{\xi}-\mu_{\xi_{s^-}^{s(n)}})(\cup
\mathcal{I})|<\ee_n/16\}\end{equation}
 Since
$\mathcal{F}(\mathcal{P}_{n-1})$ is a finite subset of
$\mathcal{F}([0,1]\setminus
 D_{X^{**}})$ and  $\mu_\xi(\partial(\cup
 \mathcal{I}))=0$,
 for all $\mathcal{I}\in\mathcal{F}([0,1]\setminus
 D_{X^{**}})$ and all
 $\xi<\omega_1$ (where $\partial(\cup
 \mathcal{I})$ is  the boundary  of $\cup\mathcal{I}$),
 the set $\{\mu_\xi:\xi\in\Gamma_{s}^{(1)}\}$
 is a relatively  weak$^*$-open nbhd of $\mu_{\xi_{s^-}^{s(n)}}$ in
$\{\mu_\xi\}_{\xi\in \Gamma_{s^-}}$. By our inductive assumption,
$\mu_{\xi_{s^-}^0}$ and  $\mu_{\xi_{s^-}^1}$  are
weak$^*$-condensation points of $\{\mu_\xi\}_{\xi\in
\Gamma_{s^-}}$ and therefore for all $s\in 2^n$ the set
$\Gamma_{s}^{(1)}$ is uncountable. Applying  Lemma \ref{fact2} we
obtain a $2^n$-tuple $(U_s)_{s\in 2^n}$ of pairwise disjoint open
subsets of $[0,1]$ and a family $(\Gamma_s^{(2)})_{s\in 2^n}$ such
that for each each $s\in 2^n$, $\Gamma_s^{(2)}$ is an uncountable
subset of $\Gamma_s^{(1)}$ and  for all $\xi\in \Gamma_s^{(2)}$,
\begin{equation}\label{eqcon2}\tau_\xi([0,1]\setminus U_s)<\theta_0/2^{n+2}\end{equation}
For every $s\in 2^n$ we set $V_{s}=U_{s}\cap V_{s^-}$.  Since
$\Gamma^{(2)}_{s}\subseteq \Gamma^{(1)}_{s}\subseteq
\Gamma_{s^-}$, using  (i), we get that for all $s\in 2^n$ and all
$\xi\in \Gamma^{(2)}_s$,
\begin{equation}\label{eqcon4}\tau_\xi([0,1]\setminus V_s)<\big(\sum_{i=0}^{n}2^{-(i+2)}\big)\theta_0\end{equation}
Moreover as $\sum_{i=0}^{n}2^{-(i+2)}<\theta_0/2$ and
$\|\tau_\xi\|>\theta_0$, we get that for all
$\xi\in\Gamma_s^{(2)}$,
\begin{equation}\label{eqcon5}\tau_\xi( V_s)>\theta_0/2\end{equation}
Since for all $\zeta<\xi<\omega_1$, we have $\mu_\xi\geq \tau_\xi$
and $\tau_\xi\perp \mu_\zeta$, by Lemma \ref{la}, we get that
$\mu_{f_{\xi}-f_{\zeta}}\geq\tau_{\xi}$.  Therefore
\begin{equation}\label{eqcon7}\mu_{f_{\xi}-f_{\zeta}}(V_t)>\tau_{\xi}(V_t)>\theta_0/2,\end{equation}
 for all  $s\in 2^n$ and
$\zeta<\xi$ in $\Gamma_s^{(2)}$.

Let $(s_i)_{i=1}^{2^n}$ be  the lexicographical enumeration of
$2^n$. Using Lemma \ref{bio}, Proposition \ref{ps} and a finite
induction on $1\leq i\leq 2^n$, we will choose for every $1\leq
i\leq 2^n$, the set $\Gamma_{s_i}$, the function $H_{s_i}$, the
pair of ordinals $(\xi_{s_i}^0,\xi_{s_i}^1)$ and the family
$\mathcal{J}_{s_i}$ satisfying (ii)-(vi). Suppose that for some
$1\leq k<2^n$, $(H_{s_i})_{i\leq k}$
 have been chosen so that $(H_{s_i})_{i\leq k}$ is an $((\ee_i^k),(\delta_i)_{i=0}^{k-1})$-biorthogonal
 sequence, where $\ee_i^k=(\sum_{r=1}^{k-1+1}2^{-r})\ee/2^i$.
  Then by Lemma \ref{bio} there are  $\delta_m>0$ and $\epsilon>0$ such that for every $H\in V_2^0$ with
  $\|H\|_\infty<\epsilon$, the sequence
  $H_{s_1},..., H_{s_{m-1}},H$ is an $((\ee^{k+1}_i)_{i=1}^{k+1},(\delta_i)_{i=0}^{m})$-biorthogonal sequence.
   Clearly, we  may
  suppose that $\epsilon<\ee_n$. For each $\xi<\omega_1$, let
  $f_\xi\in \mathcal{F}$ such that $\mu_\xi=\mu_{f_\xi}$.
Since $D_\mathcal{F}$ is countable, by Proposition \ref{ps}, we
have that $(\mathcal{F},\|\cdot\|_\infty)$ is separable and so
there is an uncountable subset $\Gamma^{(3)}_{t_{k+1}}$ of
$\Gamma^{(2)}_{t_{k+1}}$ such that for all $\zeta,\xi$ in
$\Gamma^{(3)}_{t_{k+1}}$,
 \begin{equation}\label{eqcon8}\|f_{\xi}-f_{\zeta}\|_\infty<\epsilon/8\end{equation}
Applying also Lemma \ref{ld}, for the family
$\mathcal{F}=\{f_\xi\}_{\xi\in\Gamma^{(3)}_{s_{k+1}}}$
 we pass to a further uncountable subset $\Gamma^{(4)}_{s_{k+1}}$
of $\Gamma^{(3)}_{s_{k+1}}$ such that for  every
$\zeta,\xi\in\Gamma^{(4)}_{s_{k+1}}$,
\begin{equation}\label{eqcon6}\|\mu_{f_{\xi}-f_{\zeta}}^d\|<(\epsilon/128)^2\end{equation}
We set $\Gamma_{s_{k+1}}=\Gamma^{(4)}_{s_{k+1}}$ and we choose
$\xi_{s_{k+1}}^0<\xi_{s_{k+1}}^1$ in $\Gamma_{s_{k+1}}$ such that
$\mu_{\xi_{s_{k+1}}^0}$ and $\mu_{\xi_{s_{k+1}}^1}$ are
weak$^*$-condensation points of the set
$\{\mu_\xi\}_{\xi\in\Gamma_{s_{k+1}}}$.  We put
$F=f_{\xi^0_{s_{k+1}}}- f_{\xi^1_{s_{k+1}}}$.
 Since
for all $\xi<\omega_1$, $\|f_\xi\|_{V_2}=1$, we have that
$\|F\|_\infty\leq 2$. Moreover by (\ref{eqcon7})-(\ref{eqcon6}),
we have that
\begin{equation}\mu_F(V_{s_{k+1}})>\theta_0/2, \;\;\|F \|_\infty<\epsilon/12,\;\;\text{and}
\;\;\|\mu_F^d\|<(\epsilon/128)^2\end{equation} Let $(f_n)_n$ be a
sequence in $X$ pointwise converging to $F$ with
$\|f_n\|_{V_2}\leq \|F\|_{V_2}$ (see Remark \ref{Rem5}). By
Proposition \ref{ldom},  there exist  a convex block sequence
$(h_n)_n$ of $(f_n)_n$ and $\mathcal{J}\in \mathcal{A}$ such that
 setting $H=h_l-h_k$, for sufficiently large $k<l$, we have that
 \begin{enumerate}
 \item[(a)] $\|H\|_{V_2}\leq 2\|F\|_{V_2}\leq 4$ and $\|H\|_\infty\leq
4\|F\|_\infty+\epsilon/2\leq \epsilon$. \item[(b)]
$\cup\mathcal{J}\subseteq V_{{s_{k+1}}}$ and
$\textit{v}_2^2(H,\mathcal{J})>\theta_0/2$. \item[(c)] For all
$\mathcal{I}\in \mathcal{A}$,
\[\textit{v}_2^2(H,\mathcal{I})\leq 4\mu_F(\cup
\mathcal{I})+32\|f\|_{V_2}\sqrt{\|\mu_F^d\|}+\epsilon/2 \leq
8(\mu_{\xi^0_{s_{k+1}}}
+\mu_{\xi^1_{s_{k+1}}})(\cup\mathcal{I})+\epsilon\]\end{enumerate}
We set $H_{s_{k+1}}=H$ and $J_{s_{k+1}}=J$ and the inductive step
 is completed. Finally, using Proposition \ref{det}, we choose a
sufficiently dense finite subset $\mathcal{P}_n\subseteq
[0,1]\setminus D_{\mathcal{F}}$ determining the quadratic
variation of $<\{H_s\}_{s\in 2^n}>$ which completes  the proof of
the inductive step.\end{proof}

Lemma \ref{lemind}, Proposition \ref{gen1} and  Proposition
\ref{corgen1} yield the following.
\begin{prop}\label{last1}  Let $X$ be a subspace of $V_2^0$ and
suppose that $ X^{**}$ contains an uncountable family
$\mathcal{F}$ such that $D_\mathcal{F}=\cup_{f\in\mathcal{F}}D_f$
is countable and
$\mathcal{M}_{\mathcal{F}}=\{\mu_f\}_{f\in\mathcal{F}}$ is
non-separable. Then $X$ contains a subspace isomorphic to the
space $S^2$. \end{prop}
\begin{lem}\label{subtree}  Let $\mathcal{G}=((g_s),(I_s,J_s))_{s\in 2^{<\nn}}$
be a tree family such that  $TF=\overline{<\{g_s\}_s>}$ and for
every $n\geq 0$, let $K_n=\cup_{s\in 2^n}I_s$. Then for every
$f\in TF^{**}$, $supp \;\mu_f\subseteq K$, where
$K=\cap_{n=0}^\infty K_n$.
\end{lem}
\begin{proof} Let $f\in TF^{**}$ and let
$(f_m)_m$ be a sequence in $<\{g_s\}_s>$ pointwise convergent to
$f$ and such that $\|f_n\|_{V_2}\leq \|f\|_{V_2}$. For each $n\geq
0,$ let $P_n:TF^{**}\to G_n$ be the natural projection onto the
finite dimensional space $G_n=<\{g_s\}_{|s|< n}>$ (where
$G_0=\{0\}$). Let also $h^n_m=f_m-P_n(f_m)$ and $h_n=f- P_n(f)$.
Since  $P_n$ is $w^*-w^*$ continuous, the sequence $(h^n_m)_m$ is
pointwise convergent to $h_n$ and so $supp\;h_n\subseteq K_n$.
Since $\mu_{h_n}=\mu_f$ and $supp\;\mu_{h_n}\subseteq supp\;h_n$,
we conclude that $supp\; \mu_f\subseteq K_n$ for all $n\geq 0$.
\end{proof}

\begin{prop}\label{last3} The set  $\mathcal{M}^c_{TF^{**}}=\{\mu_f: f\in TF^{**}\cap C[0,1]\}$
is a non-separable subset
 of  $\mathcal{M}[0,1]$. Therefore the space $S^2$ is embedded
 into $TF$.
\end{prop}
\begin{proof} Let $\mathcal{G}=((g_s),(I_s,J_s))_{s\in 2^{<\nn}}$
be a tree family such that  $TF=\overline{<\{g_s\}_s>}$. Let also
$(T_\sigma)_{\sigma\in 2^{\nn}}$ be the almost disjoint family of
dyadic subtrees in $2^{<\nn}$ defined in Remark \ref{remalmost
disjoint}. For each $\sigma\in 2^{\nn}$, we set
$\mathcal{G}_\sigma=((g_s),(I_s,J_s))_{s\in T_\sigma}$. As we have
already mentioned in the definition of the space $TF$,
$\mathcal{G}_\sigma$ is also a tree family and hence the space
$X_\sigma=\overline{<\{g_s\}_{s\in T_\sigma}>}$ is a copy of $TF$.
Therefore $c_0$ is embedded into $X_\sigma$ which gives that
$(X_\sigma^{**}\setminus X)\cap C[0,1]\neq \emptyset$ (cf. Remark
\ref{rem6}). So for each $\sigma\in 2^{\nn}$, we can pick a
$f_\sigma\in X_\sigma^{**}\setminus X \cap C[0,1].$ Setting
$T_\sigma=(t^\sigma_s)_{s\in 2^\nn}$ and
$K_\sigma=\cap_n\cup_{s\in 2^n}I_{t^\sigma_s}$, by Lemma
\ref{subtree} we have that  $supp\; \mu_{f_\sigma}\subseteq
K_\sigma$. Since $(T_\sigma)_{\sigma\in 2^\nn}$ is an almost
disjoint family, we get that $(K_\sigma)_{\sigma\in 2^\nn}$ is a
disjoint family of compact subsets of $[0,1]$ and so
$\{\mu_{f_\sigma}\}_{\sigma\in 2^\nn}$ consists  of pairwise
singular positive  measures. As $\{f_\sigma\}_{\sigma\in
2^\nn}\subseteq TF^{**}\cap C[0,1],$ we conclude that
$\mathcal{M}^c_{TF^{**}}$ is non-separable. Finally, that $S^2$ is
embedded into $X$, follows by Proposition \ref{last1}, for
$\mathcal{F}=TF^{**}\cap C[0,1]$.
\end{proof}

\begin{prop} \label{last2}Let $X$ be a subspace of $V_2^0$ such that  the space  $S^2$ is embedded into
 $X$. Then the set $\mathcal{M}_{X^{**}}=\{\mu_f: f\in X^{**}\}$ is a non-separable subset of  $\mathcal{M}[0,1]$.
\end{prop}
\begin{proof} Let $T$ be  an isomorphic embedding of $S^2$ into $X$
 and let $f_s=T(e_s)$, where $(e_s)_s$ be the usual basis of
$S^2$. From \cite{AMP} we have that  for each $\sigma\in 2^\nn$,
the sequence $(\sum_{k=0}^nf_{\sigma|k})_n$ is pointwise
converging to a function $f_\sigma\in (X^{**}\setminus X)\cap
C[0,1]$.  Hence there exist an uncountable subset $\Sigma\subseteq
2^\nn$ and $\delta>0$ such that for all $\sigma\in \Sigma$,
$dist(f_\sigma, X)>\delta$.
 We will
 show that the set $\{\mu_{f_\sigma}:\sigma\in \Sigma\}\subseteq \mathcal{M}[0,1]$ is
non-separable. Indeed, otherwise, we can choose   a
norm-condensation point $\mu\in\mathcal{M}[0,1]$ of
$\{\mu_\sigma:\sigma\in 2^\nn\}$. Fix also  a positive integer
$m\in\nn$ and $\ee>0$. Then for uncountably many
$\sigma\in\Sigma$, we have that
\begin{equation}\label{mm1}\|\mu_{f_\sigma}-\mu\|\leq
\ee/m\end{equation} Let $\sigma_1,...,\sigma_m\in  \Sigma$
satisfying (\ref{mm1}) and let $n_0\in\nn$ be such that for all
$n\geq n_0$ and $1\leq i< j\leq m$, $\sigma_i|n\perp\sigma_j|n$.
We set $g_{\sigma_i}=\sum_{n\geq n_0}f_{\sigma_i|n}$. Since
$f_{\sigma_i}-g_{\sigma_i}=\sum_{n<n_0}f_{\sigma_i|n}\in X$, we
have that $dist(g_{\sigma_i}, X)=dist(f_{\sigma_i}, X)>\delta$ and
$\mu_{f_{\sigma_i}}=\mu_{g_{\sigma_i}}$.  For every $n\in\nn$, let
$F_n^i=\sum_{k=n_0}^{n_0+n}f_{\sigma_i|k}$. Then
\[\|F_n^i\|_{V_2}\leq
\|T\|\Big\|\sum_{k=n_0}^{n_0+n}e_{\sigma_i|k}\Big\|_{S^2}\leq
\|T\|\] Applying   Proposition \ref{c011}, for the continuous
function $g_{\sigma_i}\in X^{**}\setminus X$,
 the sequence
$(F_n^i)_n$ and $\ee_n=\ee/m 2^n$, we obtain   a convex block
sequence $(h^i_n)_n$ of $(F^i_n)_n$ such that the functions
$G^i_n=h^i_{2n+1}-h^i_{2n}$, satisfy  the following. (i)
$\delta<\|G_n^i\|_{V_2}\leq 2\|T\|$, (ii) $\|G^i_n\|_\infty<\ee/m
2^{2n}$ and (iii) for every $\mathcal{I}\in \mathcal{A}$,
$\textit{v}_2^2(G^i_n,\mathcal{I})\leq 4\mu_{f_{\sigma_i}}(\cup
\mathcal{I})+\ee/m 2^{2n}$.

 By  (ii) and Lemma \ref{bio}, we can
choose $n_1<...<n_m$ such that the finite sequence
$(G_{n_i}^i)_{i=1}^m$ is $\ee/m-$biorthogonal. By the definition
of $G_{n}^i$, we have that
 $G_{n_i}^i=\sum_{s\in F_i}\lambda_s f_s$, where $F_i$ is a finite subset of $\{\sigma_i|n: n\in\nn\}$.
Hence \begin{equation}\label{T1}\delta< \|G_{n_i}^i\|_{V_2}\leq
\|T\|\Big\|\sum_{s\in F_i}\lambda_s e_s\Big\|_{S^2}
\leq\|T\|\max_{s\in F_i}|\lambda_s|\end{equation} Let  $s_i\in
F_i$ be such that $|\lambda_{s_i}|=\max_{s\in F_i}|\lambda_s|$.
Then by (\ref{T1}), $|\lambda_{s_i}|\geq \delta/\|T\|$ and so,
since the set $\{s_i:1\leq i\leq m\}$ is an antichain of
$2^{<\nn}$, we get that
\begin{equation}\label{g1}\Big\|\sum_{i=1}^m G_{n_i}^i\Big\|_{V_2}\geq\frac{1}{ \|T^{-1}\|}
 \Big\|\sum_{i=1}^m\sum_{s\in F_i}\lambda_s e_s\Big\|_{S^2}\geq \frac{1}{\|T^{-1}\|}\sqrt{\frac{m\delta^2}{\|T\|^2}}\geq
\frac{\delta\sqrt{m}}{\|T^{-1}\|\|T\|}\end{equation} By Lemma
\ref{lul2} and (iii), for every $\mathcal{I}\in \mathcal{A}$ we
have that
\[\begin{split}\textit{v}_2^2(\sum_{i=1}^m G^i_{n_i},\mathcal{I})&\leq
\sum_{i=1}^m\textit{v}_2^2(G_{n_i}^i,\mathcal{I}^{(i)})
+(4\|T\|+1)\ee\leq 4\sum_{i=1}^m\mu_{f_{\sigma_i}}(\cup
\mathcal{I}^{(i)})+(4\|T\|+2)\ee\\&\leq 4\mu(\cup
\mathcal{I})+\ee+(4\|T\|+2)\ee\leq
4\|\mu\|+(4\|T\|+3)\ee\end{split}\] Therefore, letting $\ee\to 0$,
$4\|\mu\|\geq \Big\|\sum_{i=1}^m G_{n_i}^i\Big\|_{V_2}$ and so by
(\ref{g1}) , we get a contradiction.\end{proof}

We are finally ready to prove the main results of the paper.
\begin{thm} \label{main1}Let $X$ be a subspace of $V_2^0$.
Then  the space $S^2$ is embedded into $X$ if and only if
$\mathcal{M}_{X^{**}}$ is non-separable.
\end{thm}
\begin{proof} By Proposition \ref{last2}, if $S^2$ is embedded
into $X$ then $\mathcal{M}_{X^{**}}$ is non-separable. Conversely
suppose that $\mathcal{M}_{X^{**}}$ is non-separable. Then we
distinguish two cases. If $X^*$ is separable then by Proposition
\ref{P1}, the set $D_{X^{**}}$ is countable and hence by
Proposition \ref{last1}, for $\mathcal{F}=X^{**}$, we get that
$S^2$ is embedded into $X$. In the case   $X^*$ is non-separable,
by \cite{AAK}, the space $TF$ is embedded into $X$. By Proposition
\ref{last3}, we have that $S_2$ is embedded into $TF$ and hence
into $X$. \end{proof}

\begin{thm} \label{main2}Let $X$ be a subspace of $V_2^0$.
Then $c_0$ is embedded into $X$ if and only if  $X^{**}$ is non
separable.
\end{thm}
\begin{proof} Suppose that $X^{**}$ is non
separable (the other direction is obvious).  If
 $X^{*}$ is non-separable then as we have already
mentioned the space $TF$ and hence $c_0$ is embedded into $X$. So
assume that $X^*$ is separable. We distinguish the following
cases. If $\mathcal{M}_{X^{**}}$ is non-separable then the result
follows by Theorem \ref{main1}. Otherwise, $\mathcal{M}_{X^{**}}$
is separable and  so by Proposition \ref{submain}, $c_0$ is again
embedded into $X$.
\end{proof}

\end{document}